\documentclass[12pt,twoside]{amsart}
\usepackage{mathrsfs,amsthm,amscd,amsmath,amssymb}
\usepackage[colorlinks,linkcolor=green,citecolor=blue, pdfstartview=FitH]{hyperref}
\usepackage{amscd}
\usepackage{actuarialsymbol}
\usepackage{amsfonts}
\usepackage{graphicx}
\usepackage{caption, subcaption}

\title
[Comonotonic and moment matching approximations for sums of lognormal random variables]
{Comonotonic and moment matching approximations for sums of lognormal random variables}
\author{Chunle Huang}
\address{Chunle Huang, School of Mathematics, Hunan University, Changsha, 410082, China.}
\email{2019044@hnu.edu.cn; 402961544@qq.com}

\newtheorem{thm}{Theorem}[section]
\newtheorem{lem}[thm]{Lemma}
\newtheorem{cor}[thm]{Corollary}

\newtheorem{prop}[thm]{Proposition}

\theoremstyle{definition}
\newtheorem{defn}[thm]{Definition}

\newtheorem{ex}[thm]{Example}
\makeatletter
\let\uppercasenonmath\@gobble
\makeatother
\usepackage{graphicx}  
\usepackage{caption}   
\begin{document}
\bibliographystyle{amsalpha+}
\maketitle
\begin{abstract} 
In this paper, based on the concept of weighted distribution, we introduce a kind of new approximations for sums of lognormal random variables, such that they are both comonotonic and moment matching. Numerical results show that the approximation performance of the newly presented approximations is, overall, comparable to the classical comonotonic approximations, but in terms of the right tail of the distribution of the original sum our approximations perform better than the classical comonotonic ones. Another contribution of this article is the establishment of the step-weighting theory for continuous random variables. 
\end{abstract}
\tableofcontents 
\section{Introduction} \label{62773}
In this paper we consider approximations for the distribution and risk measures of a random variable $S$ defined by $S = \sum^n_{i = 1}\alpha_ie^{Z_i}$ where $\alpha_i > 0$ are positive real numbers and $(Z_1, ..., Z_n)^{'}$ is a multivariate random vector. In practice, the quantities $\alpha_i$ can be used to denote future deterministic payments or saving amounts while the random variables $Z_i$ can be used to descirbe financial risks. For instance, if $Z_i$ denotes the stochastic logreturn over the period $[i, n]$ for each $i = 1, ..., n$, then the random variable $S$ can be interpreted as the accumulated value at time $n$ of a series of future deterministic saving amounts $\alpha_i$. On the other hand, if $-Z_i$ denotes the stochastic logreturn over the period $[0, i]$, then $e^{Z_i}$ can be interpreted as the stochastic discount factor over the period $[0, i]$. In this case, the random variable $S$ can be interpreted as the stochastic present value of a series of future deterministic payments $\alpha_i$. Given that Gaussian models are widely used in finance for modeling asset returns, we will mainly assume that ${Z} = (Z_1, ..., Z_n)^{'}$ is a random vector that follows a multivariate normal distribution. In this case, the random variable $S$ becomes a sum of lognormal random variables and it is thus impossible to derive analytical expressions for its distribution function, let along its risk measures. Therefore, it is of great interest to consider its approximations.

In the literature, several techniques for approximating the distribution and risk measures of $S$ have been proposed. One technique is based on the idea of matching the first two moments of $S$ using fixed popular distributions. Since $S$ is the sum of lognormal random variables the first choice is to use a moment matching lognormal approximation for the distribution of $S$. Another choice is to use a moment matching reciprocal Gamma approximation since the present value of a constant continuous perpetuity with lognormal return process, which can be considered as a limiting case of the random variable $S$ as defined above, has a reciprocal Gamma distribution, see \cite{D1990} and \cite{M1997} for the details. For a detailed discussion of the moment matching methods we refer to \cite{KGDD2008}, \cite{KPW2012}, \cite{MP1998}, \cite{MR2000} and \cite{VHD2005b}. 

Another technique is using the concept of comonotonicity. Indeed, based on an idea of \cite{RS1995} in an Asian option context, the authors in the celebrated paper \cite{KDG2000} propose to approximate the distribution function of $S$ by the distribution function of the random variable $S^l$ defined by 
$S^l \overset{\mathrm{def}}{=} \text{E}[S\, |\, \Lambda] = \sum^n_{i = 1}\alpha_i \text{E}[e^{Z_i} \,|\, \Lambda]$
for an appropriate conditioning random variable $\Lambda$. Loosely speaking, this approach allows one to transform the multivariate randomness of $(Z_1, ..., Z_n)$ to the univariate randomness of $\Lambda$. Moreover, an appropriate choice of $\Lambda$ will lead to a comonotonic random vector $(\alpha_1\text{E}[e^{Z_1} | \Lambda], ..., \alpha_n\text{E}[e^{Z_n} | \Lambda])$, which means that all the components $\alpha_i\text{E}[e^{Z_i} | \Lambda]$ are non-decreasing functions of the conditioning random variable $\Lambda$. Risk measures related to the distribution function of $S$ are then approximated by the corresponding risk measures of $S^l$. These approximations are straightforward to calculate, especially for distortion risk measures, taking into account the additivity property of sums of comonotonic random variables. For an extensive overview on the concept of comonotonicity we refer to \cite{DDGK2006} and \cite{Dhaene et al. (2002a)}. For applications of this concept in actuarial science and finance we refer to \cite{DD2007}, \cite{Dhaene et al. (2006)}, \cite{DHL2008}, \cite{DWY2000} and \cite{VDG2003}. 

Moment matching approximations are fairly easy to understand while comonotonic approximations facilitate the computation of risk measures, both of which are very important from practical point of view when considering approximations for the distribution and risk measures of the sum $S$. Therefore, a natural question arises whether there exist approximations for the distribution and risk measures of $S$ with excellent fitting performance such that they are both moment matching and comonotonic. Following the fundamental work of \cite{Dhaene et al. (2002a)}, \cite{Dhaene et al. (2002b)}, \cite{DVG2005}, \cite{KDG2000} and \cite{VCD2008} and making use of the new concept of step-weighted distribution we will show that this can be done without much effort. In fact, our method holds true for any sum of continuous random variables which clearly includes sums of lognormal random variables as a special case. Numerical results show that the approximation performance of the newly presented approximations is, overall, comparable to the classical comonotonic approximations, but in terms of the right tail of the distribution of $S$, a particularly important concept in actuarial science, our approximations perform better than the classical comonotonic ones, see Section \ref{627604} for the details. Given that the fitting performance of the classical comonotonic approximations is better than the classical moment matching approximations for a wide range of parameter values, see \cite{VHD2005b} and \cite{VGD2004} for instance, we can conclude that our approximations generally outperform the classical moment matching ones. 

Another contribution of this article is the proposal of the step-weighting theory. The concept of weighted distribution can be traced back to the late 1970s in the statistical literature, see the original work of Patil, Ord and Rao in \cite{PR1976} and \cite{PR1978}, and has proven to be a useful tool in the study of risk measures after the pioneering work by Furman and Landsman in \cite{FL2005} and \cite{FL2008} and Furman and Zitikis in \cite{FZ2008a}, \cite{FZ2008b} and \cite{FZ2009}. The reader is also referred to the recent work of Denuit in \cite{Denuit2019} and \cite{Denuit2020} for a related introduction to the weighted transform and its properties. In general, the weighted distribution of a random variable is not easy to handle. Therefore, the literature mainly focuses on weighted distributions of random variables associated with monotone weight functions. In particular, much attention has been paid to the case of the identity weight function and the corresponding weighted distribution is known as the size-biased transform, see \cite{AGK2019} and \cite{Qin2017} for instance. In this paper, we will consider the case where the weight function is a step function and the corresponding weight distribution will be called step-weighted distribution. We will show that such weighted distributions possess a number of very nice properties. For instance, for any continuous random variable $X$, the inverse distribution function of the step-weighted distribution can easily be expressed in terms of the inverse distribution function of $X$ based on an auxiliary function, see Theorem \ref{1231219}. Another elegant property is that the operation of comonotonic summing interchanges with the operation of step-weighting for any continuous random variables, see Theorem \ref{0103383}.

The remainder of this paper is organized as follows. Section 2 collects some preliminary results. Section 3 is devoted to the study of the step-weighted distribution of continuous random variables. In Section 4, we consider several risk measures of the step-weighted distribution. In Section 5, we study the step-weighted distribution of comonotonic sums. In Section 6, we give comonotonic and moment matching approximations for sums of continuous random variables. Section 7 conducts numerical analysis and Section 8 concludes the paper. 
\section{Preliminaries}
In this section we recall some preliminary results which will be used in the following. 
\subsection{Comonotonicity}
In this subsection, we recall the concept of comonotonicity. 
\begin{defn}
A random vector $\underline{X} = (X_1, ..., X_n)$ is said to be comonotonic if $\underline{X} \stackrel{d}{=} (F^{-1}_{X_1}(U), ..., F^{-1}_{X_n}(U))$, where $ \stackrel{d}{=}$ stands for equality in distribution and $U \sim U(0, 1)$.
\end{defn}
\noindent By definition the components of a comonotonic random vector $\underline{X}$ are jointly driven by a single random variable transformed by the non-decreasing functions $F^{-1}_{X_1}, ..., F^{-1}_{X_n}$. Comonotonicity has been extensively discussed in \cite{Dhaene et al. (2002a)} and \cite{Dhaene et al. (2002b)}. For its applications in finance and actuarial science we recommend the reader to refer to \cite{DDV2011}, \cite{DVG2005}, \cite{DWY2000} and \cite{WD1998}. The following result is useful when considering comonotonic random vectors. 

\begin{lem} [Theorem 3 of \cite{Dhaene et al. (2002a)}] \label{626115} 
A random vector $\underline{X}$ is comonotonic if and only if there exist a random variable $Z$ and non-decreasing functions $f_i, i = 1, ..., n$, such that $$\underline{X} \stackrel{d}{=} (f_1(Z), ..., f_n(Z)).$$ \end{lem}
\subsection{Convex bounds for sums of random variables} Let $(X_1, ..., X_n)$ be a random vector and $(X_1^c, ..., X_n^c)$ be a comonotonic modification of $(X_1, ..., X_n)$, which means that $(X_1^c, ..., X_n^c)$ is a comonotonic random vector with the same marginal distributions as $(X_1, ..., X_n)$. Let $S = \sum^n_{i = 1}X_i$ be the sum of the components of $(X_1, ..., X_n)$ and $S^c$ be the comonotonic sum of $X_1, ..., X_n$, that is, $$S^c = X_1^c + ... + X_n^c.$$ For any conditioning random variable $\Lambda$ we denote the sum of the conditional expectations $E[X_i \,|\, \Lambda], i = 1, ..., n,$ by $$S^l = E[X_1 \,|\, \Lambda] + ... + E[X_n \,|\, \Lambda] .$$ Note that if there exists a conditioning random variable $\Lambda$ such that all the conditional expectations $E[X_i \,|\, \Lambda], i = 1, ..., n$, are non-decreasing functions of $\Lambda$ then the lower bound $S^l$ will naturally become a comonotonic sum. We introduce the following famous result, concerning convex order bounds for sums of dependent random variables. A detailed proof for this result can be found, for instance, in \cite{Dhaene et al. (2002a)} and \cite{KDG2000}. 

\begin{thm} [Convex bounds for sums of random variables] \label{65401} For any random vector $(X_1, ..., X_n)$ and any conditioning random variable $\Lambda$, we have that 
\begin{equation} \label{65403}
 S^l \leq_{cx} S \leq_{cx} S^c
\end{equation}
\end{thm}
\noindent From Theorem \ref{65401} the random sums $S^c$ and $S^l$ give convex bounds for the original sum $S$. On the other hand, in case the dependence between the random variables $X_1, ..., X_n$ is unknown or too cumbersome to work with so that it is impossible to derive the distribution function of $S$ exactly, we can regard $S^c$ and $S^l$ as approximations to $S$. Numerical results show that the lower bound $S^l$ provides good approximations to $S$ in general cases, while the performance of the upper bound $S^c$ is not generally good, depending on the dependence between the random variables $X_1, ..., X_n$. For a detailed discussion of the random sums $S^c$ and $S^l$ as approximations to $S$ and their applications in actuarial science and finance we refer to \cite{Dhaene et al. (2002a)}, \cite{Dhaene et al. (2002b)}, \cite{DVG2005}, \cite{KDG2000} and \cite{VCD2008}. 

\subsection{Comonotonic approximations for sums of lognormal random variables}
Consider the random sum $S = \sum^n_{i = 1}\alpha_ie^{Z_i}$ where $\alpha_1, ..., \alpha_n$ are positive real numbers and $(Z_1, ..., Z_n)^{'}$ is a random vector following a multivariate normal distribution with 
\begin{equation} \label{324118}
\text{Cov}[Z_i, Z_j] = \min(i, j)\sigma^2, \,\, i, j = 1, 2, ..., n, \sigma >0.
\end{equation}
Following the papers  \cite{Dhaene et al. (2002a)}, \cite{Dhaene et al. (2002b)}, \cite{DVG2005}, \cite{KDG2000} and \cite{VCD2008} we regard the convex upper and lower bounds $S^c$ and $S^l$ as approximations of $S$. For the upper bound $S^c$, we have 
\begin{equation} \label{612530}
S^c = \sum^n_{i = 1}\alpha_ie^{E[Z_i] + \sigma_{Z_i}\Phi^{-1}(U)}
\end{equation}
where $U$ is a random variable uniformly distributed on the unit interval $(0, 1)$. 
For the lower bound $S^l$, we choose the conditioning random variable $\Lambda$ such that 
\begin{equation} \label{324076}
\Lambda = \sum^n_{j = 1}\lambda_jZ_j
\end{equation}
where $\lambda_1, ..., \lambda_n$ are real numbers which will be specified later. We denote the mean and variance of $\Lambda$ by $\text{E}[\Lambda]$ and $\sigma_{\Lambda}^2$, respectively. Then $S^l = \text{E} [S\, |\, \Lambda]$ can be written as
\begin{equation}\label{219079}
S^l = \sum^n_{i = 1}\alpha_i \exp\Big(\text{E}[Z_i] + \frac{1}{2}(1 - r_i^2)\sigma_{Z_i}^2 + r_i\sigma_{Z_i}\Phi^{-1}(V)\Big)
\end{equation}
where $V = \Phi(\frac{\Lambda - \text{E}[\Lambda]}{\sigma_{\Lambda}})$ is a random variable uniformly distributed on $(0, 1)$ and $r_i$ is the correlation coefficient between $Z_i$ and $\Lambda$. 
Provided that all coefficients $r_i$ are positive, the terms in $S^l$ will be non-decreasing functions of the same random variable $V$. Hence, by Lemma \ref{626115}, $S^l$ will be a comonotonic sum in this case. This implies that the quantiles and conditional tail expectations related to $S^l$ can be computed by summing the corresponding risk measures for
the marginals involved. 
\begin{ex} [The Taylor-based approximation] \label{621193}
In the celebrated papers \cite{Dhaene et al. (2002a)}, \cite{Dhaene et al. (2002b)} and \cite{KDG2000} the authors propose to choose the conditioning random variable $\Lambda$ as a linear combination of $Z_1, ..., Z_n$ with the coefficients $\lambda_j$ given by 
\begin{equation} \label{324119}
\lambda_j^{TB} = \alpha_je^{\text{E}[Z_j]}, \,\, j = 1, ..., n.
\end{equation}
This choice makes $\Lambda$ a linear transformation of a first-order approximation of the sum $S$. 
In this case $S^l$ will be a good approximation to $S$, provided $\sigma$ is sufficiently small. More importantly, under the condition (\ref{324118}) this choice makes $S^l$ a comonotonic sum. 
We call the approximation based on (\ref{324119}) the Taylor-based approximation. 
\end{ex}

\begin{ex} [The maximal variance approximation] \label{621204}
Apart from the Taylor-based approximation, the authors in \cite{VCD2008}, \cite{VDG2005a} and \cite{VHD2005b} propose to choose the conditioning random variable $\Lambda$ as a linear combination of $Z_i$'s with the coefficients $\lambda_j$ given by 
\begin{equation} \label{324135}
\lambda^{MV}_j = \alpha_j\text{E}[e^{Z_j}], \,\, j = 1, ..., n. 
\end{equation} 
This choice will ensure that the first order approximation of $\text{Var}[S^l]$ is maximized. 
Moreover, under the condition (\ref{324118}) it defines a comonotonic approximation to $S$. We call the approximation $S^l$ based on (\ref{324135}) the maximal variance approximation. 
\end{ex} 

\section{Step-weighted distribution} \label{26112} 
This section is devoted to the study of step-weighted distribution, a new concept which will be seen useful when considering comonotonic approximations for sums of continuous random variables. 
\subsection{Definition of step-weighted distribution}
In this subsection, we introduce a new class of weight functions, named step-weight functions, and the corresponding weighted distributions will be called step-weighted distributions. We will show that such weighted distributions possess a number of nice properties. For instance, for any continuous random variable $X$, the inverse distribution function of the step-weighted version of $X$ can be easily expressed in terms of the inverse distribution function of $X$ based on an auxiliary function, see Theorem \ref{1231219} for the details. Throughout this paper, $m \in \mathbb{N}^+$ will always denote a positive integer. Then we make the following definition. 

\begin{defn} \label{52378}
Let $A = (a_1, ..., a_m)^{'}$ and $Q = (q_1, ..., q_m)^{'}$ be two numerical vectors with $\min(a_1, ..., a_m) \geq 0, \max(a_1, ..., a_m) > 0$ and $0 = q_0 < q_1 < q_2 < ... < q_m = 1$. Let $X$ be a continuous random variable which means that the distribution function $F_X(x)$ of $X$ is continuous on $(-\infty, \infty)$. Then we call $\omega_{A, Q, X}(x)$ the step-weight function of the random variable $X$ associated with the numerical vectors $A$ and $Q$, abbreviated as $\omega_{A, Q}$ when there is no ambiguity, if it is defined by 
$$\omega_{A, Q, X}(x) = \begin{cases} a_1, & F^{-1+}_{X}(q_0) \leq x < F_X^{-1}(q_1), \\  
a_k, & F_X^{-1}(q_{k - 1}) \leq x < F_X^{-1}(q_k), \,\, k = 2, ..., m - 1, \\
a_m, & F_X^{-1}(q_{m - 1}) \leq x \leq F^{-1}_{X}(q_m), \\ 
0, & \text{otherwise}. 
\end{cases}$$
\end{defn}

The following proposition shows that the expectation $E[\omega_{A, Q}(X)]$ always exists for any continuous random variable $X$ and it is independent of the distribution of $X$. 
\begin{prop} \label{123084}
Let $A = (a_1, ..., a_m)^{'}$ and $Q = (q_1, ..., q_m)^{'}$ be two numerical vectors with $\min(a_1, ..., a_m) \geq 0, \max(a_1, ..., a_m) > 0$ and $0 = q_0 < q_1 < q_2 < ... < q_m = 1$. Let $X$ be a continuous random variable and $\omega_{A, Q}$ the step-weight function of $X$. Then it holds that $$0 < E[\omega_{A, Q}(X)] = \sum^m_{i = 1}a_i(q_{i} - q_{i - 1}) < \infty.$$ In particular, the expectation $E[\omega_{A, Q}(X)]$ is independent of the distribution of $X$. 
\end{prop}
\begin{proof} 
By assumption, the distribution function $F_X(x)$ of the random variable $X$ is continuous, which implies that $$F_{X}(F_X^{-1}(q)) = q \,\, \text{for any} \,\, q \in (0, 1),$$ see \cite{DDGK2006} and \cite{Dhaene et al. (2002a)} for instance, where $F_X^{-1}(q)$ is the inverse distribution function of $X$. So, from the definition of step-weight function we find that 
$$E[\omega_{A, Q}(X)] 
= \sum_{i = 1}^m\int_{F^{-1}_X(q_{i - 1})}^{F^{-1}_{X}(q_{i})}\omega_{A, Q}(x)dF_X(x) 
= \sum^m_{i = 1}a_i(q_{i} - q_{i - 1}).$$
This completes the proof. 
\end{proof}

Now, we introduce the concept of step-weighted distribution as follows. 
\begin{defn} \label{529103}
Let $A = (a_1, ..., a_m)^{'}$ and $Q = (q_1, ..., q_m)^{'}$ be two numerical vectors with $\min(a_1, ..., a_m) \geq 0, \max(a_1, ..., a_m) > 0$ and $0 = q_0 < q_1 < q_2 < ... < q_m = 1$. Let $X$ be a continuous random variable with distribution function $F_X(x)$. We call $X_{A, Q}$ the step-weighted version of $X$ associated with the numerical vectors $A$ and $Q$ if the distribution function $F_{X_{A, Q}}(x)$ of $X_{A, Q}$ is determined by $$dF_{X_{A, Q}}(x) = \frac{\omega_{A, Q}(x)}{E[\omega_{A, Q}(X)]}dF_X(x), \,\, \text{for} \,\, x \in \mathbb{R}.$$ Furthermore, the distribution of $X_{A, Q}$ will be called the step-weighted distribution of $X$ associated with the numerical vectors $A$ and $Q$. 
\end{defn}

\subsection{Properties of step-weighted distribution} 
Let $r$ be a positive integer. In the following proposition we compute the r-th order origin moment $E[(X_{A, Q})^r]$ of the step-weighted distribution. We can see from the proposition that the origin moment $E[(X_{A, Q})^r]$ of the step-weighted distribution can be expressed as the weighted average of the conditional expectations of $X^r$ on the intervals $(F^{-1}_X(q_{k - 1}), F^{-1}_X(q_k))$ with the weights given by 
\begin{equation} \label{529108}
w_k := \frac{a_k(q_k - q_{k - 1})}{\sum^m_{i = 1}a_i(q_{i} - q_{i - 1})}, \,\, k = 1, ..., m.
\end{equation}
\begin{prop} \label{213201}
Let $A = (a_1, ..., a_m)^{'}$ and $Q = (q_1, ..., q_m)^{'}$ be two numerical vectors with $\min(a_1, ..., a_m) \geq 0, \max(a_1, ..., a_m) > 0$ and $0 = q_0 < q_1 < q_2 < ... < q_m = 1$. Let $X$ be a continuous random variable with inverse distribution function $F_X^{-1}(t)$ and $w_k$ be the weights defined by (\ref{529108}). Then for any positive integer $r$ it holds that  
$$E[(X_{A, Q})^r] = \sum^m_{k = 1}w_kE[X^r \,|\, F^{-1}_X(q_{k - 1}) \leq X \leq F^{-1}_X(q_k)].$$
\end{prop}
\begin{proof}
From Definition \ref{52378}, Definition \ref{529103} and Proposition \ref{123084} we find that 
\begin{align*} 
E[(X_{A, Q})^r] &= \int^{\infty}_{-\infty}x^rdF_{X_{A, Q}}(x) \\ 
&= \sum^m_{k = 1}\frac{a_k}{E[\omega_{A, Q, X}(X)]}\int^{F^{-1}_X(q_k)}_{F^{-1}_X(q_{k - 1})}x^rdF_X(x) \\ 
&= \sum^m_{k = 1}\frac{a_k(q_k - q_{k - 1})}{\sum^m_{i = 1}a_i(q_{i} - q_{i - 1})}E[X^r \,|\, F^{-1}_X(q_{k - 1}) \leq X \leq F^{-1}_X(q_k)] \\ 
&= \sum^m_{k = 1}w_kE[X^r \,|\, F^{-1}_X(q_{k - 1}) \leq X \leq F^{-1}_X(q_k)]. 
\end{align*}
This completes the proof. 
\end{proof}

Next, we consider inverse distribution functions of step-weighted distribution. To this end, we need to establish the following lemmas, including Lemma \ref{112074}, Lemma \ref{1231168} and Lemma \ref{112087}, which will be frequently used in the following sections.

\begin{lem} \label{112074}
Let $X$ be a random variable, not necessarily continuous, and $\omega$ a nonnegative weight function such that $0 < E[\omega(X)] < \infty$. Define the function 
\begin{equation} \label{120277}
G_{X, \omega}(t) = \frac{1}{E[\omega(X)]}\int^t_0 \omega(F^{-1}_X(u))du, \,\, t \in [0, 1]
\end{equation}
where $F_X^{-1}(u)$ is the inverse distribution function of $X$. Then we have that 
\begin{itemize} 
\item $G_{X, \omega}(0) = 0$ and $G_{X, \omega}(1) = 1$, 
\item $G_{X, \omega}(t)$ is increasing in $t$, 
\item $G_{X, \omega}(t)$ is continuous in $t$. 
\end{itemize}
In other words, $G_{X, \omega}(t)$ is a continuous distortion function on the unit interval $[0, 1]$. 
\end{lem}
\begin{proof} 
It is easy to check that $G_{X, \omega}(0) = 0$ and $G_{X, \omega}(1) = 1$. $G_{X, \omega}(t)$ is increasing in $t$ because the weight function $\omega(x)$ is assumed to be nonnegative. $G_{X, \omega}(t)$ is continuous because $\omega(F^{-1}_X(u))$ is an integrable function on the unit interval $[0, 1]$. 
\end{proof}
In particular, we have 
\begin{lem} \label{1231168}
Let $A = (a_1, ..., a_m)^{'}$ and $Q = (q_1, ..., q_m)^{'}$ be two numerical vectors with $\min(a_1, ..., a_m) \geq 0, \max(a_1, ..., a_m) > 0$ and $0 = q_0 < q_1 < q_2 < ... < q_m = 1$. Let $X$ be a continuous random variable, $\omega_{A, Q}$ the step-weight function of $X$ associated with the numerical vectors $A$ and $Q$, and $a = \sum^m_{i = 1}a_i(q_{i} - q_{i - 1})$. Then it holds that 
\begin{equation} \label{529202}
G_{X, \omega_{A, Q}}(t) = a^{-1}\Big(\sum^{k - 1}_{i = 1}a_{i}(q_{i} - q_{i - 1}) + a_k(t - q_{k - 1})\Big), \,\, t \in [q_{k - 1}, q_{k}], k = 1, ..., m 
\end{equation}
with $\sum_{i = 1}^{0} = 0$ by convention. In particular, the distortion function $G_{X, \omega_{A, Q}}(t)$ is independent of the distribution of $X$, which will be denoted by $G_{A, Q}(t)$ in the following. 
\end{lem}
\begin{proof} 
The statement of Lemma \ref{1231168} follows immediately from direct calculations given the definition of step-weight function and the continuity assumption of $X$. 
\end{proof}

\begin{lem} \label{112087}
Let $X$ be a random variable, not necessarily continuous, $\omega(x)$ a nonnegative weight function such that $0 < E[\omega(X)] < \infty$ and ${X_\omega}$ the weighted version of $X$ associated with $\omega$. Then we have that 
\begin{itemize} 
\item $F_{X_\omega}(x)  = G_{X, \omega}(F_X(x)) \,\, \text{for any} \,\, x \in \mathbb{R}$, 
\item $F_{X_\omega}^{-1}(p) = F_X^{-1}(G^{-1}_{X, \omega}(p))$ for any $p \in (0, 1)$, 
 \item $F_{X_\omega}^{-1+}(p) = F_X^{-1+}(G^{-1+}_{X, \omega}(p))$ for any $p \in (0, 1)$, 
\end{itemize}
where $G_{X, \omega}(t)$ is the continuous distortion function as given in Lemma \ref{112074}. 
\end{lem}

\begin{proof} 
By the definition of weighted distribution and the well-known identify $$X  \stackrel{d}{=} F_X^{-1}(U),$$ see \cite{Dhaene et al. (2002a)} for instance, where $ \stackrel{d}{=}$ stands for equality in distribution and $U$ is a random variable uniformly distributed on the unit interval $(0, 1)$, we have that 
\begin{align*} 
F_{X_\omega}(x) &= \frac{E[\mathbb{I}\{X \leq x\}\omega(X)]}{E[\omega(X)]} = \frac{E[\mathbb{I}\{F_X^{-1}(U) \leq x\}\omega(F_X^{-1}(U))]}{E[\omega(X)]} \\ 
& = \frac{\int^1_0\mathbb{I}\{F^{-1}_X(u) \leq x\}\omega(F_X^{-1}(u))du}{E[\omega(X)]} 
= G_{X, \omega}(F_X(x)), \,\, \text{for} \,\, x \in \mathbb{R}.  
\end{align*}
Combining this and the following well-known equivalences for continuous distribution functions 
\begin{itemize} 
\item $G_{X, \omega}^{-1}(q) \leq x \iff q \leq G_{X, \omega}(x)$ for any $x \in \mathbb{R}, q \in [0, 1]$ 
\item $x \leq G^{-1+}_{X, \omega}(q) \iff G_{X, \omega}(x) \leq q$ for any $x \in \mathbb{R}, q \in [0, 1]$
\end{itemize} 
see \cite{DDGK2006} for the details, we find that 
\begin{align*} 
F_{X_\omega}^{-1}(p) &= \inf_{x \in \mathbb{R}}\{x \,|\, F_{X_\omega}(x) \geq p\}= \inf_{x \in \mathbb{R}}\{x \,|\, G_{X, \omega}(F_X(x))\geq p\}\\ 
&= \inf_{x \in \mathbb{R}}\{x \,|\, F_X(x)\geq G_{X, \omega}^{-1}(p)\} = F_X^{-1}(G^{-1}_{X, \omega}(p))
\end{align*}
and 
\begin{align*} 
F_{X_\omega}^{-1+}(p) &= \sup_{x \in \mathbb{R}}\{x \,|\, F_{X_\omega}(x) \leq p\} = \sup_{x \in \mathbb{R}}\{x \,|\, G_{X, \omega}(F_X(x))\leq p\}\\ 
&= \sup_{x \in \mathbb{R}}\{x \,|\, F_X(x)\leq G_{X, \omega}^{-1+}(p)\} = F_X^{-1+}(G^{-1+}_{X, \omega}(p))
\end{align*}
for any $p \in (0, 1)$. This completes the proof. 
\end{proof}
Based on the above preliminary lemmas we obtain the following theorem, expressing the inverse distribution functions of the step-weighted distribution in terms of the inverse distribution functions of the original distribution. 
\begin{thm} \label{1231219}
Let $A = (a_1, ..., a_m)^{'}$ and $Q = (q_1, ..., q_m)^{'}$ be two numerical vectors with $\min(a_1, ..., a_m) \geq 0, \max(a_1, ..., a_m) > 0$ and $0 = q_0 < q_1 < q_2 < ... < q_m = 1$. Let $X$ be a continuous random variable and $X_{A, Q}$ the step-weighted version of $X$ associated with the numerical vectors $A$ and $Q$. Then it holds that 
\begin{itemize} 
\item $F_{X_{A, Q}}(x)  = G_{A, Q}(F_X(x)) \,\, \text{for any} \,\, x \in \mathbb{R}$, 
\item $F_{X_{A, Q}}^{-1}(p) = F_X^{-1}(G^{-1}_{A, Q}(p))$ for any $p \in (0, 1)$, 
 \item $F_{X_{A, Q}}^{-1+}(p) = F_X^{-1+}(G^{-1+}_{A, Q}(p))$ for any $p \in (0, 1)$, 
\end{itemize} 
where $G_{A, Q}(t)$ is the continuous distortion function as given in Lemma \ref{1231168}. 
\end{thm}
\begin{proof} 
Theorem \ref{1231219} follows immediately from Lemma \ref{1231168} and Lemma \ref{112087}. 
\end{proof}
As a natural generalization of the inverse distribution function, the authors in \cite{KDG2000} introduced the $\alpha$-inverse distribution function in order to present a closed expression for the stop-loss premiums of comonotonic sums. Since then, this concept has proven to be useful in some actuarial contexts, see \cite{DDR2022} and \cite{MH2011}, amongst others. Based on the above results, we can also express the $\alpha$-inverse distribution function of the step-weighted distribution in terms of the $\alpha$-inverse distribution function of the original distribution as follows. 

\begin{cor} \label{1231229}
Let $A = (a_1, ..., a_m)^{'}$ and $Q = (q_1, ..., q_m)^{'}$ be two numerical vectors with $\min(a_1, ..., a_m) > 0$ and $0 = q_0 < q_1 < q_2 < ... < q_m = 1$. Let $X$ be a continuous random variable and $X_{A, Q}$ be the step-weighted version of $X$ associated with the numerical vectors $A$ and $Q$. Then it holds that 
\begin{equation} \label{529262}
F_{X_{A, Q}}^{-1(\alpha)}(p) = F_X^{-1(\alpha)}(G^{-1}_{A, Q}(p)), \,\, \text{for any} \,\, p \in (0, 1), \alpha \in [0, 1],
\end{equation}
where 
\begin{equation} \label{1231234}
G_{A, Q}^{-1}(p) = a_k^{-1}{a}(p - u_{k - 1}) + q_{k - 1}, \,\, p \in [u_{k - 1}, u_{k}], k = 1, ..., m 
\end{equation}
with $a = \sum^m_{i = 1}a_i(q_{i} - q_{i - 1})$, $u_0 = 0$ and $u_k = a^{-1}\sum^{k}_{i = 1}a_i(q_{i} - q_{i - 1}), k = 1, ..., m$. 
\end{cor}
\begin{proof} 
If the numerical vector $A = (a_1, ..., a_m)^{'}$ satisfies $\min(a_1, ..., a_m) > 0$ then from Lemma \ref{1231168} we find that the distortion function $G_{A, Q}(t)$ is continuous and strictly increasing on $(0, 1)$. This implies that $$G^{-1}_{A, Q}(p) = G^{-1+}_{A, Q}(p), \,\, \text{for any} \,\, p \in (0, 1),$$ see \cite{Dhaene et al. (2002a)} for instance. Thus, from Theorem \ref{1231219} it follows that 
\begin{align*}
F_{X_{A, Q}}^{-1(\alpha)}(p) &= \alpha F_{X_{A, Q}}^{-1}(p) + (1 - \alpha)F_{X_{A, Q}}^{-1+}(p) \\ 
&= \alpha F_X^{-1}(G_{A, Q}^{-1}(p)) + (1 - \alpha)F_X^{-1+}(G_{A, Q}^{-1+}(p)) \\ 
&= \alpha F_X^{-1}(G_{A, Q}^{-1}(p)) + (1 - \alpha)F_X^{-1+}(G_{A, Q}^{-1}(p)) \\ 
&= F_X^{-1(\alpha)}(G_{A, Q}^{-1}(p))
\end{align*} 
for any $p \in (0, 1)$ and $\alpha \in [0, 1]$. The expression (\ref{1231234}) for the inverse distribution function $G_{A, Q}^{-1}(p)$ is easy to check. This completes the proof. 
\end{proof}

\section{Risk measures of step-weighted distribution}
This section is devoted to the study of the risk measures VaR, TVaR and stop-loss premium of the step-weighted distribution. 
\subsection{VaR of step-weighted distribution} Recall that the Value-at-Risk (VaR) of a random variable $X$ at level $p \in (0, 1)$ is defined by $\text{VaR}_p[X] = F^{-1}_X(p)$. By Theorem \ref{1231219} we obtain the following result regarding VaR of step-weighted distribution. 
\begin{prop} \label{120367}
Let $A = (a_1, ..., a_m)^{'}$ and $Q = (q_1, ..., q_m)^{'}$ be two numerical vectors with $\min(a_1, ..., a_m) \geq 0, \max(a_1, ..., a_m) > 0$ and $0 = q_0 < q_1 < q_2 < ... < q_m = 1$. Let $X$ be a continuous random variable and $X_{A, Q}$ the step-weighted version of $X$ associated with the numerical vectors $A$ and $Q$. Then for any $p \in (0, 1)$ it holds that $$\text{VaR}_p[X_{A, Q}] = \text{VaR}_{G_{A, Q}^{-1}(p)}[X].$$
\end{prop}
\begin{proof} 
Proposition \ref{120367} follows immediately from Theorem \ref{1231219}. 
\end{proof}
\subsection{TVaR of step-weighted distribution} Recall that the Tail-Value-at-Risk (TVaR) of a random variable $X$ at level $p \in (0, 1)$ is defined by $\text{TVaR}_p[X] = \frac{1}{1 - p}\int^1_{p}F^{-1}_X(q)dq$. By Theorem \ref{1231219} and Lemma \ref{1231168} we obtain the following result concerning TVaR of step-weighted distribution. 
\begin{prop} \label{14371}
Let $A = (a_1, ..., a_m)^{'}$ and $Q = (q_1, ..., q_m)^{'}$ be two numerical vectors with $\min(a_1, ..., a_m) \geq 0, \max(a_1, ..., a_m) > 0$ and $0 =q_0 < q_1 < q_2 < ... < q_m = 1$. Let $X$ be a continuous random variable and $X_{A, Q}$ be the step-weighted version of $X$. Then for any $p \in (0, 1)$ it holds that 
\begin{align*} 
\text{TVaR}_p[X_{A, Q}] = 
\frac{a_{l}(1 - p^{*})}{a(1 - p)}\text{TVaR}_{p^{*}}[X] + \sum^{m - 1}_{k = l}\frac{(a_{k + 1} - a_{k})(1 - q_k)}{a(1 - p)}\text{TVaR}_{q_k}[X]
\end{align*}
with $p^{*} = G^{-1}_{A, Q}(p) \in [q_{l - 1}, q_{l}]$ for some $l \in \{1, ..., m\}$ and $\sum_{k = m}^{m - 1} = 0$ by convention. 
\end{prop}
\begin{proof} Let $p \in (0, 1)$ be any fixed probability level and $p^{*} = G^{-1}_{A, Q}(p)$. First, we assume that $p^{*} > 0$ and $p^{*} \in (q_{l - 1}, q_{l}]$ for some $l \in \{1, ..., m\}$. Then by Theorem \ref{1231219} and Lemma \ref{1231168} we have that 
\begin{align*} 
\text{TVaR}_p[X_{A, Q}] &= \frac{1}{1 - p}\int^1_pF^{-1}_{X_{A, Q}}(q)dq \\ 
&= \frac{1}{1 - p}\int^1_{p^{*}}F^{-1}_X(t)dG_{A, Q}(t) \\ 
&= \frac{a_l}{a(1 - p)}\int^{q_l}_{p^{*}}F_X^{-1}(t)dt +  \sum^{m}_{k = l + 1}\frac{a_k}{a(1 - p)}\int^{q_k}_{q_{k - 1}}F^{-1}_X(t)dt \\ 
&= \frac{a_{l}(1 - p^{*})}{a(1 - p)}\text{TVaR}_{p^{*}}[X] + \sum^{m - 1}_{k = l}\frac{(a_{k + 1} - a_{k})(1 - q_k)}{a(1 - p)}\text{TVaR}_{q_k}[X]. 
\end{align*}
One can easily check that it is still valid in case $p^{*} = 0$. This completes the proof. 
\end{proof}

\subsection{Stop-loss premium of step-weighed distribution} 
Recall that the stop-loss premium with retention $d$ of a random variable $X$ is defined by $E[(X - d)_{+}]$, with the notation $(x - d)_{+} = \max(x - d, 0)$. By Theorem \ref{1231219} and Proposition \ref{14371} we obtain the following result concerning stop-loss premium of step-weighted distribution. 
\begin{prop} \label{0108390}
Let $A = (a_1, ..., a_m)^{'}$ and $Q = (q_1, ..., q_m)^{'}$ be two numerical vectors with $\min(a_1, ..., a_m) \geq 0, \max(a_1, ..., a_m) > 0$ and $0 = q_0 < q_1 < q_2 < ... < q_m = 1$. Then, for any continuous random variable $X$ the stop-loss premiums of the step-weighted distribution associated with the numerical vectors $A$ and $Q$ are given by  
\begin{align*}
E[(X_{A, Q} - d)_{+}] &= {a}^{-1}\Big[a_lE[(X - d)_{+}] + d\Big(a_l(1 - q_{l}) - \sum^{m - 1}_{k = l}a_{k + 1}(q_{k + 1} - q_{k})\Big)  \\ 
&\quad + \sum^{m - 1}_{k = l}(a_{k + 1} - a_{k})(1 - q_k)\text{TVaR}_{q_k}[X]\Big], \,\, d \in (F^{-1+}_{X_{A, Q}}(0), F^{-1}_{X_{A, Q}}(1))
\end{align*}
with $F_X(d) \in (q_{l - 1}, q_l]$ for some $l \in \{1, ..., m\}$ and $\sum_{k = m}^{m - 1} = 0$ by convention. 
\end{prop}
\begin{proof} 
Let $p = F_{X_{A, Q}}(d)$ and $p^{*} = G_{A, Q}^{-1}(p)$. By the definition of weighted distribution $$(F^{-1+}_{X_{A, Q}}(0), F^{-1}_{X_{A, Q}}(1)) \subseteq (F_{X}^{-1+}(0), F_{X}^{-1}(1)).$$ 
Therefore, we can assume that $F_X(d) \in (q_{l - 1}, q_l]$ for some $l \in \{1, ..., m\}$. If $a_l > 0$, then by Theorem \ref{1231219} we find that $$
p^{*} = G_{A, Q}^{-1}(p) = G_{A, Q}^{-1}(F_{X_{A, Q}}(d)) = G_{A, Q}^{-1}(G_{A, Q}(F_X(d))) = F_X(d)  \in (q_{l - 1}, q_l].$$
Hence, by the well-known identity of stop-loss premiums 
\begin{align*} 
E[(X - d)_{+}] &= (\text{TVaR}_{F_{X}(d)}[X] - d)(1 - F_{X}(d)), 
\end{align*}
see \cite{Dhaene et al. (2006)} for instance, and Proposition \ref{14371} we have that 
\begin{align*} 
E[(X_{A, Q} - d)_{+}] 
&= (\text{TVaR}_p[X_{A, Q}] - d)(1 - p) \\ 
&= {a}^{-1}\Big[a_l(1 - p^{*})\text{TVaR}_{p^{*}}[X] + \sum^{m - 1}_{k = l}(a_{k + 1} - a_{k})(1 - q_k)\text{TVaR}_{q_k}[X]\Big] \\ 
&\quad - \Big(1 - {a}^{-1}(\sum^{l - 1}_{k = 1}a_k(q_k - q_{k - 1}) + a_l(p^{*} - q_{l - 1})\Big)d \\ 
&= {a}^{-1}\Big[a_lE[(X - d)_{+}] + d\Big(a_l(1 - q_{l}) - \sum^{m - 1}_{k = l}a_{k + 1}(q_{k + 1} - q_{k})\Big)  \\ 
&\quad + \sum^{m - 1}_{k = l}(a_{k + 1} - a_{k})(1 - q_k)\text{TVaR}_{q_k}[X]\Big], \,\, d \in (F^{-1+}_{X_{A, Q}}(0), F^{-1}_{X_{A, Q}}(1))
\end{align*}
with $\sum_{k = m}^{m - 1} = 0$ by convention. One can easily check that the equation is still valid if $a_l = 0$. This completes the proof. 
\end{proof}

\section{Step-weighted distribution of comonotonic sums}\label{26538} 
In this section, we study the step-weighted distribution of comonotonic sums of continuous random variables. A number of new results will be presented. 
\subsection{A fundamental result}
In this subsection, we investigate the step-weighted distribution of comonotonic sums. Let $(X_1, ..., X_n)^{'}$ be a random vector of dimension $n$ and $\omega$ be a weight function. Then, we call $((X_1)_{\omega}, ..., (X_n)_{\omega})^{'}$ the weighted version of the random vector $(X_1, ..., X_n)^{'}$ associated with the weight function $\omega$. As usual, we denote the comonotonic sum of the components of $(X_1, ..., X_n)^{'}$ by $$S^c = X_1^c + ... + X_n^c$$ and the weighted version of $S^c$ associated with the weight function $\omega$ by $(S^c)_{\omega}$. The sum and the comonotonic sum of the components of the weighted random vector $((X_1)_{\omega}, ..., (X_n)_{\omega})^{'}$ will be denoted by 
$$S_{\omega} = (X_1)_{\omega} + ... + (X_n)_{\omega}$$ and $$(S_{\omega})^c = ((X_1)_{\omega})^c + ... + ((X_n)_{\omega})^c$$ respectively. Then we have 
\begin{thm} \label{0103383}
Let $A = (a_1, ..., a_m)^{'}$ and $Q = (q_1, ..., q_m)^{'}$ be two numerical vectors with $\min(a_1, ..., a_m) \geq 0, \max(a_1, ..., a_m) > 0$ and $0 = q_0 < q_1 < q_2 < ... < q_m = 1$. Then for any continuous random variables $X_1, ..., X_n$ we have that $$(S^c)_{\omega} \stackrel{d}{=} (S_{\omega})^c$$
where $\stackrel{d}{=}$ stands for equality in distribution and $X_{\omega}$ denotes the step-weighted version of a continuous random $X$ associated with the step-weight function $\omega(x) = \omega_{A, Q, X}(x)$. 
\end{thm}
\begin{proof} 
By the well-known additivity of quantiles for comonotonic sums we have that $$F_{S^c}^{-1}(p) = F_{X_1}^{-1}(p) + ... + F_{X_n}^{-1}(p), \,\, \text{for any} \,\, p \in (0, 1)$$ see \cite{DDGK2006} and \cite{Dhaene et al. (2002a)} for instance. Since each $X_i$ is assumed to be continuous, the inverse distribution function $F_{X_i}^{-1}(p)$ is strictly increasing on $(F^{-1+}_{X_i}(0), F^{-1}_{X_i}(1))$, $i = 1, ..., n$. Therefore, $F_{S^c}^{-1}(p)$ is strictly increasing on $(F^{-1+}_{S^c}(0), F^{-1}_{S^c}(1))$, which in turn implies that $S^c$ is a continuous random variable. By Theorem \ref{1231219} and the well-known additivity of quantiles for comonotonic sums we find that 
$$F_{(S^c)_{\omega}}^{-1}(p) = F_{S^c}^{-1}(G_{A, Q}^{-1}(p)) = F_{(X_1)_{\omega}}^{-1}(p) + ... + F_{(X_n)_{\omega}}^{-1}(p) = F_{(S_{\omega})^c}^{-1}(p)$$
for any $p \in (0, 1)$. Thus, it follows that  
$$(S^c)_{\omega} \stackrel{d}{=} F_{(S^c)_{\omega}}^{-1}(U) = F_{(S_{\omega})^c}^{-1}(U) \stackrel{d}{=} (S_{\omega})^c$$
where $U$ is a random variable uniformly distributed on $(0, 1)$. This ends the proof. 
\end{proof}

From Theorem \ref{0103383} we can conclude that the step-weighted version $(S^c)_{\omega}$ of the comonotonic sum $S^c$ of continuous random variables $X_1, ..., X_n$ is equal in distribution to the comonotonic sum $(S_{\omega})^c$ of the step-weighted versions of $X_1, ..., X_n$. In other words, the operation of comonotonic summing interchanges with the operation of step-weighting for continuous random variables. 

Here we present an example to illustrate the main findings of Theorem \ref{0103383}. 
\begin{ex}
Let $A = (\frac{1}{2}, \frac{11}{6}, \frac{5}{6})^{'}$ and $Q = (\frac{1}{4}, \frac{1}{2}, 1)^{'}$. Consider the sum $S = X_1 + X_2$ where $X_1$ and $X_2$ are random variables uniformly distributed over the unit interval $(0, 1)$. Let $S^c = X_1^c + X_2^c$ where $(X_1^c, X_2^c)^{'}$ is a comonotonic modification of $(X_1, X_2)^{'}$. By the well-known additivity of quantiles for comonotonic sums one can easily check that $S^c \sim U(0, 2)$. In this case, the step-weighted distribution of $S^c$ associated with the step-weight function $\omega = \omega_{A, Q}$ is given by $$F_{(S^c)_\omega}(x)\begin{cases} \frac{1}{4}x & x \in [0, \frac{1}{2}) \\ 
\frac{11}{12}(x - \frac{1}{2}) + \frac{1}{8} & x \in [\frac{1}{2}, 1) \\ 
\frac{5}{12}(x - 1) + \frac{7}{12} & x \in [1, 2]
\end{cases}.$$
On the other hand, for the random variables $X_i, i = 1, 2$, we have $$F_{(X_i)_\omega}(x) = \begin{cases} 
\frac{1}{2}x & x \in [0, \frac{1}{4}) \\ 
\frac{11}{6}(x - \frac{1}{4}) + \frac{1}{8} & x \in [\frac{1}{4}, \frac{1}{2}) \\ 
\frac{5}{6}(x - \frac{1}{2}) + \frac{7}{12} & x \in [\frac{1}{2}, 1]
\end{cases}.$$ This combining the well-known additivity of quantiles for comonotonic sums implies that $$F^{-1}_{(S_\omega)^c}(p) = \begin{cases} 
4p & p \in [0, \frac{1}{8}] \\ 
\frac{12}{11}(p - \frac{1}{8}) + \frac{1}{2} & p \in (\frac{1}{8}, \frac{7}{12}] \\ 
\frac{12}{5}(p - \frac{7}{12}) + 1 & p \in (\frac{7}{12}, 1]
\end{cases}.$$ Now it is easy to check that $$F_{(S^c)_\omega}(x) = F_{(S_\omega)^c}(x) \,\, \text{for} \,\, x \in \mathbb{R}.$$
\end{ex}

\subsection{Quantiles of step-weighted distribution of comonotonic sums}In the following theorem, we consider the $\alpha$-inverse distribution function $F^{-1(\alpha)}_{(S^c)_{\omega}}(p)$ of the step-weighted version $(S^c)_{\omega}$ of the comonotonic sum $S^c$. We prove that under mild conditions the $\alpha$-inverse distribution function of the step-weighted comonotonic sum $(S^c)_{\omega}$ at probability level $p$ can be expressed as the sum of the $\alpha$-inverse distribution functions of the marginal distributions involved at probability level $G_{A}^{-1}(p)$. This generalizes Theorem 6 of \cite{Dhaene et al. (2002a)} to the case of step-weighted distributions. 

\begin{thm} \label{27543}
Let $A = (a_1, ..., a_m)^{'}$ and $Q = (q_1, ..., q_m)^{'}$ be two numerical vectors with $\min(a_1, ..., a_m) > 0$ and $0 = q_0 < q_1 < q_2 < ... < q_m = 1$. Then for any continuous random variables $X_1, ..., X_n$ the $\alpha$-inverse distribution function $F^{-1(\alpha)}_{(S^c)_{\omega}}(p)$ of the step-weighted comonotonic sum $(S^c)_{\omega}$ with $\omega = \omega_{A, Q}$ is given by 
$$F^{-1(\alpha)}_{(S^c)_{\omega}}(p) = \sum^n_{i = 1}F_{X_i}^{-1(\alpha)}(G_{A, Q}^{-1}(p)), \,\, \text{for any} \,\, p \in (0, 1), \alpha \in [0, 1].$$
\end{thm}
\begin{proof} 
From Theorem \ref{0103383}, Corollary \ref{1231229}, and the well-known additivity of quantiles for comonotonic sums it follows that  
\begin{align*}
F_{(S^c)_{\omega}}^{-1(\alpha)}(p) = F_{(S_{\omega})^c}^{-1(\alpha)}(p) = \sum^n_{i = 1}F_{(X_i)_{\omega}}^{-1(\alpha)}(p) = \sum^n_{i = 1}F_{X_i}^{-1(\alpha)}(G_{A, Q}^{-1}(p))
\end{align*}
for any $p \in (0, 1), \alpha \in [0, 1]$. This ends the proof. 
\end{proof}
\subsection{TVaR of step-weighted distribution of comonotonic sums} In the next theorem, we consider TVaR of the step-weighted comonotonic sum $(S^c)_{\omega}$. We prove that the TVaR of $(S^c)_{\omega}$ can be expressed as a linear combination of the TVaR's of the involved marginal terms $\text{TVaR}_{p^{*}}[X_i]$ and $\text{TVaR}_{q_k}[X_i], i = 1, ..., n, k = 1, ..., m$. 
\begin{thm} \label{65351}
Let $A = (a_1, ..., a_m)^{'}$ and $Q = (q_1, ..., q_m)^{'}$ be two numerical vectors with $\min(a_1, ..., a_m) \geq 0, \max(a_1, ..., a_m) > 0$ and $0 = q_0 < q_1 < q_2 < ... < q_m = 1$. Then for any continuous random variables $X_1, ..., X_n$ and for any $p \in (0, 1)$ the TVaR of the step-weighted comonotonic sum $(S^c)_{\omega}$ with $\omega = \omega_{A, Q}$ is given by 
\begin{align*} 
\text{TVaR}_p[(S^c)_{\omega}] = 
\frac{a_{l}(1 - p^{*})}{a(1 - p)}\sum^n_{i = 1}\text{TVaR}_{p^{*}}[X_i] + \sum^{n}_{i = 1}\sum^{m - 1}_{k = l}\frac{(a_{k + 1} - a_{k})(1 - q_k)}{a(1 - p)}\text{TVaR}_{q_k}[X_i]
\end{align*}
with $p^{*} = G^{-1}_{A, Q}(p) \in [q_{l - 1}, q_{l}]$ for some $l \in \{1, ..., m\}$ and $\sum_{k = m}^{m - 1} = 0$ by convention. 
\end{thm}
\begin{proof} 
By Theorem \ref{0103383}, Proposition \ref{14371}, and the well-known additivity of TVaR's for comonotonic sums we have that 
\begin{align*}
TVaR_p[(S^c)_{\omega}] &= TVaR_p[(S_{\omega})^c] \\ 
&= \sum^n_{i = 1} TVaR_p[(X_i)_{\omega}] \\ 
&= \frac{a_{l}(1 - p^{*})}{a(1 - p)}\sum^n_{i = 1}\text{TVaR}_{p^{*}}[X_i] + \sum^{n}_{i = 1}\sum^{m - 1}_{k = l}\frac{(a_{k + 1} - a_{k})(1 - q_k)}{a(1 - p)}\text{TVaR}_{q_k}[X_i]
\end{align*} 
with $p^{*} = G^{-1}_{A, Q}(p) \in [q_{l - 1}, q_{l}]$ for some $l \in \{1, ..., m\}$. This completes the proof. 
\end{proof}

\subsection{Stop-loss premium of step-weighted distribution of comonotonic sums}
In the following theorem, we consider stop-loss premium of the step-weighted comonotonic sum $(S^c)_{\omega}$. We prove that the stop-loss premium of $(S^c)_{\omega}$ can be obtained from the stop-loss premiums and TVaR's of the marginal distributions involved. Our result generalizes Theorem 7 of \cite{Dhaene et al. (2002a)} to the case of step-weighted distributions. 
\begin{thm} \label{0108458}
Let $A = (a_1, ..., a_m)^{'}$ and $Q = (q_1, ..., q_m)^{'}$ be two numerical vectors with $\min(a_1, ..., a_m) \geq 0, \max(a_1, ..., a_m) > 0$ and $0 = q_0 < q_1 < q_2 < ... < q_m = 1$. Let $X_1, ..., X_n$ be a series of continuous random variables, $S^c$ the comonotonic sum of $X_1, ..., X_n$ and $(S^c)_{\omega}$ the step-weighted version of $S^c$ associated with the step-weight function $\omega = \omega_{A, Q}$. Then the stop-loss premiums of $(S^c)_{\omega}$ are given by 
\begin{align*}
E[((S^c)_{\omega} - d)_{+}] &= {a}^{-1}\Big[a_l\sum^n_{i = 1}E[(X_i - d_i)_{+}] + d\Big(a_l(1 - q_{l}) - \sum^{m - 1}_{k = l}a_{k + 1}(q_{k + 1} - q_{k})\Big)  \\ 
&\quad + \sum^{m - 1}_{k = l}(a_{k + 1} - a_{k})(1 - q_k)\sum^n_{i = 1}\text{TVaR}_{q_k}[X_i]\Big], \,\, d \in (F^{-1+}_{(S^c)_{\omega}}(0), F^{-1}_{(S^c)_{\omega}}(1))
\end{align*}
with $F_{S^c}(d) \in (q_{l - 1}, q_l]$ for some $l \in \{1, ..., m\}$, $\sum_{k = m}^{m - 1} = 0$ by convention, and the $d_i$ given by $$d_i = F_{X_i}^{-1(\alpha_d)}(F_{S^c}(d)), \,\, i = 1, ..., n$$ where $\alpha_d \in [0, 1]$ is determined by $$F_{S^c}^{-1(\alpha_d)}(F_{S^c}(d)) = d.$$
\end{thm}
\begin{proof} 
By Proposition \ref{0108390} we have that 
\begin{align} \label{18469}
E[((S^c)_{\omega} - d)_{+}] &= {a}^{-1}\Big[a_lE[(S^c - d)_{+}] + d\Big(a_l(1 - q_{l}) - \sum^{m - 1}_{k = l}a_{k + 1}(q_{k + 1} - q_{k})\Big)  \\ 
&\quad + \sum^{m - 1}_{k = l}(a_{k + 1} - a_{k})(1 - q_k)\text{TVaR}_{q_k}[S^c]\Big], \,\, d \in (F^{-1+}_{(S^c)_{\omega}}(0), F^{-1}_{(S^c)_{\omega}}(1)) \nonumber 
\end{align}
with $F_{S^c}(d) \in (q_{l - 1}, q_l]$ for some $l \in \{1, ..., m\}$ and $\sum_{k = m}^{m - 1} = 0$ by convention. By the definition of weighted distribution $$(F^{-1+}_{(S^c)_{\omega}}(0), F^{-1}_{(S^c)_{\omega}}(1)) \subseteq (F_{S^c}^{-1+}(0), F_{S^c}^{-1}(1)).$$ So, from Theorem 7 of \cite{Dhaene et al. (2002a)} we find that the stop-loss premium $E[(S^c - d)_{+}]$ of the comonotonic sum $S^c$ at retention $d$ can be expressed as  
\begin{equation} \label{18473}
E[(S^c - d)_{+}] = \sum^n_{i = 1}E[(X_i - d_i)_{+}]
\end{equation}
with the $d_i$ given by $$d_i = F_{X_i}^{-1(\alpha_d)}(F_{S^c}(d)), \,\, i = 1, ..., n$$ and $\alpha_d \in [0, 1]$ determined by $$F^{-1(\alpha_d)}_{S^c}(F_{S^c}(d)) = d.$$ Moreover, from the well-known additivity of distortion risk measures for comonotonic sums it follows that 
\begin{equation} \label{0108477}
\text{TVaR}_{q_k}[S^c] = \sum^n_{i = 1}\text{TVaR}_{q_k}[X_i] 
\end{equation}
see \cite{DDGK2006} and \cite{Dhaene et al. (2006)} for the details. Substituting (\ref{18473}) and (\ref{0108477}) into (\ref{18469}) leads to 

\begin{align*}
E[((S^c)_{\omega} - d)_{+}] &= {a}^{-1}\Big[a_l\sum^n_{i = 1}E[(X_i - d_i)_{+}] + d\Big(a_l(1 - q_{l}) - \sum^{m - 1}_{k = l}a_{k + 1}(q_{k + 1} - q_{k})\Big)  \\ 
&\quad + \sum^{m - 1}_{k = l}(a_{k + 1} - a_{k})(1 - q_k)\sum^n_{i = 1}\text{TVaR}_{q_k}[X_i]\Big], \,\, d \in (F^{-1+}_{(S^c)_{\omega}}(0), F^{-1}_{(S^c)_{\omega}}(1)). 
\end{align*} 
This completes the proof. 
\end{proof}
\section{Comonotonic and moment matching approximations} \label{626556}
In this section, based on the concept of weighted distribution, we introduce a new kind of approximations for sums of continuous random variables, such that they are both comonotonic and moment matching. The main result of this section is the following theorem. 

\begin{thm} \label{66420}
Consider the sum $S = \sum^n_{i = 1}X_i$ of continuous random variables $X_1, ..., X_n$ and let $T$ be a continuous comonotonic approximation of $S$, which means that there exist continuous random variables $Y_1, ..., Y_n$ such that $T = Y^c_1 + ... + Y^c_n.$ Assume that the origin moments $E[S^r]$ and $E[T^r]$ exist for a positive integer $r$ and denote $c = (1, E[S], ..., E[S^r])^{'}$. Let $Q = (q_1, ..., q_m)^{'}$ be any fixed numerical vector with $0 = q_0 < q_1 < q_2 < ... < q_m = 1$ and define the $(r + 1)$-dimensional column vectors $b_1, ..., b_m$ by 
\begin{equation} \label{67419}
b_j = (1, b_{1j}, ..., b_{rj})^{'}, \,\, j = 1, ..., m
\end{equation}
where 
\begin{equation} \label{67423}
b_{ij} = E[T^i \,|\, F_{T}^{-1}(q_{j - 1}) \leq T \leq F_{T}^{-1}(q_j)] < \infty, \,\, i = 1, ..., r, j = 1, ..., m.
\end{equation}
Denote the convex cone generated by the vectors $b_1, ..., b_m$ by $$\text{Cone}(b_1, ..., b_m) = \Big\{\sum^m_{j = 1}\lambda_jb_j \,|\, \lambda_j \geq 0, j = 1, ..., m\Big\}.$$ If $c \in \text{Cone}(b_1, ..., b_m)$ then one can find a comonotonic approximation $T_{\omega}$ of $S$ such that 
\begin{equation} \label{67427}
E[(T_\omega)^i] = E[S^i], \,\, i = 1, ..., r. 
\end{equation}
In other words, $T_\omega$ defines an approximation of $S$ such that the approximation is comonotonic and moment matching up to order $r$. 
\end{thm}
\begin{proof} 
By the definition of convex cone there exist nonnegative numbers $\lambda_1, ..., \lambda_m$ such that 
\begin{equation} \label{67433}
c = \lambda_1b_1 + ... + \lambda_mb_m.
\end{equation}
Let $A = (a_1, ..., a_m)^{'}$ with 
\begin{equation} \label{67437}
a_i = \lambda_i(q_i - q_{i - 1})^{-1}, \,\, i = 1, ..., m.
\end{equation}
Then one can easily check that the numerical vector $A$ satisfies $$\min(a_1, ..., a_m) \geq 0 \,\, \text{and} \,\, \max(a_1, ..., a_m) > 0.$$ 
Given that the random variables $Y_1, ..., Y_n$ are continuous we can conclude that as the comonotonic sum of $Y_1, ..., Y_n$ the comonotonic approximation $T$ must be continuous. For the numerical vectors $A$ and $Q$, we consider the step-weighted version $T_{\omega}$ of $T$ associated with the step-weight function $\omega = \omega_{A, Q}$. By Theorem \ref{0103383} we find that 
$$T_\omega = (Y^c_1 + ... + Y^c_n)_\omega \stackrel{d}{=} ((Y_1)_{\omega})^c + ... + ((Y_n)_{\omega})^c.$$ Therefore, the step-weighted version $T_\omega$ of $T$ gives a comonotonic approximation of $S$. It remains to show the equalities in (\ref{67427}). Indeed, by Proposition \ref{213201} we have that 
\begin{equation} \label{67443}
E[(T_{\omega})^i] = \sum^m_{j = 1}w_jE[T^i \,|\, F^{-1}_T(q_{j - 1}) \leq T \leq F^{-1}_T(q_j)] = \sum^m_{j = 1}w_jb_{ij}, \,\, i = 1, ..., r
\end{equation}
with 
\begin{equation} \label{67444}
w_j = \frac{a_j(q_j - q_{j - 1})}{\sum^m_{k = 1}a_k(q_{k} - q_{k - 1})}, \,\, j = 1, ..., m.
\end{equation}
Substituting (\ref{67437}) into (\ref{67444}) and taking into account equation (\ref{67433}) and the definition of the numerical vector $c$ we eventually obtain that 
\begin{equation*}
E[(T_\omega)^i] = E[S^i], \,\, i = 1, ..., r. 
\end{equation*}
This completes the proof. 
\end{proof}
As a starting approximation of the random sum $S$, the comonotonic approximation $T$ in Theorem \ref{66420}, obviously, can be chosen in an infinite number of ways. To obtain from the starting approximation $T$ a comonotonic and moment matching approximation $T_\omega$ of the random sum $S$ with excellent performance of approximation, the starting approximation $T$ should be selected in advance such that it is as close to the random sum $S$ as possible since in many cases $T_\omega$ is a slight modification of $T$, see the following numerical illustrations.  

To simplify Theorem \ref{66420}and to facilitate numerical illustrations we fix the notations in Theorem \ref{66420} such that $r = 2$ and $m = 3$ and make the following result, which immediately follows from the proof of Theorem \ref{66420}. 
\begin{cor} \label{613595}
Consider the sum $S = \sum^n_{i = 1}X_i$ of continuous random variables $X_1, ..., X_n$ and let $T$ be a continuous comonotonic approximation of $S$, which means that there exist continuous random variables $Y_1, ..., Y_n$ such that $T = Y^c_1 + ... + Y^c_n.$ Assume that $Var[S] < \infty$ and $Var[T] < \infty$ and denote $c = (1, E[S], E[S^2])^{'}$. Let $Q = (q_1, q_2, q_3)^{'}$ be any fixed numerical vector with $0 = q_0 < q_1 < q_2 < q_3 = 1$ and define the $3 \times 3$ matrix $B$ by $B = (b_1, b_2, b_3)$ where $b_1, b_2, b_3$ are $3$-dimensional column vectors defined by 
\begin{equation} \label{613419}
b_j = (1, b_{1j}, b_{2j})^{'}, \,\, j = 1, 2, 3
\end{equation}
with  
\begin{equation} \label{613423}
b_{ij} = E[T^i \,|\, F_{T}^{-1}(q_{j - 1}) \leq T \leq F_{T}^{-1}(q_j)] < \infty, \,\, i = 1, 2, j = 1, 2, 3.
\end{equation}
Denote the convex cone generated by the vectors $b_1, b_2, b_3$ by $$\text{Cone}(b_1, b_2, b_3) = \{\lambda_1b_1 + \lambda_2b_2 + \lambda_3b_3 \,|\, \lambda_j \geq 0, j = 1, 2, 3\}.$$ If $\det(B) \neq 0$ and $c \in \text{Cone}(b_1, b_2, b_3)$ then the step-weighted version $T_\omega $ of the starting comonotonic approximation $T$ associated with the step-weight function $\omega = \omega_{A, Q}$ with 
$$A = \text{diag}((q_1 - q_0)^{-1}, (q_2 - q_1)^{-1}, (q_3 - q_2)^{-1})B^{-1}c$$ 
defines a comonotonic approximation of $S$ such that 
\begin{equation*}
E[T_\omega] = E[S] \,\, \text{and} \,\, Var[T_\omega] = Var[S]. 
\end{equation*} 
\end{cor}
\begin{proof}
Corollary \ref{613595} immediately follows from the proof of Theorem \ref{66420}. 
\end{proof}

From Theorem \ref{66420} and Corollary \ref{613595} we see that in order to obtain comonotonic and moment matching approximations of the random sum $S$ one needs to judge whether the numerical vector $c$ belongs to the convex cone generated by the vectors $b_1, ..., b_m$. Fortunately, numerical results show that this is indeed the case for a wide range of parameter values, see the following numerical illustrations. Here we present a simple example to illustrate the main findings of Theorem \ref{66420} and Corollary \ref{613595}. 
\begin{ex} 
Let $S = X_1 + X_2$ where $X_1$ and $X_2$ are random variables uniformly distributed on the unit interval $(0, 1)$ with their correlation coefficient given by $r_{X_1, X_2} = \frac{1}{2}$. In this case we have $$E[S] = 1 \,\, \text{and} \,\, Var[S] = \frac{1}{4}.$$
Let $S^c = X_1^c + X_2^c$ denote the comonotonic sum of $X_1$ and $X_2$ , where $(X_1^c, X_2^c)^{'}$ is a comonotonic modification of $(X_1, X_2)^{'}$. Then $S^c$ gives a comonotonic approximation to the random sum $S$. Easy computation shows that $$E[S^c] = 1 \,\, \text{and} \,\, Var[S^c] = \frac{1}{3}.$$
Next, by Theorem \ref{66420}, we want modify the comonotonic sum $S^c$ to give a comonotonic approximation $(S^c)_{\omega}$ of $S$ such that $$E[(S^c)_{\omega}] = E[S] \,\, \text{and} \,\, Var[(S^c)_{\omega}] = Var[S].$$ To this end, let $Q = (q_1, q_2, q_3)^{'}$ be a numerical vector with $q_1 = \frac{1}{4}, q_2 = \frac{1}{2}$ and $q_3 =1$. Using the notations in Theorem \ref{66420} we find that $$b_1 = (1, \frac{1}{4}, \frac{1}{12})^{'}, b_2 = (1, \frac{3}{4}, \frac{7}{12})^{'} \,\, \text{and} \,\, b_3 = (1, \frac{3}{2}, \frac{7}{3})^{'}$$
so that $$c = \lambda_1b_1 + \lambda_2b_2 + \lambda_3b_3$$ with $\lambda_1 = \frac{3}{24}, \lambda_2 = \frac{11}{24} \,\, \text{and} \,\, \lambda_3 = \frac{10}{24}$. Define the numerical vector $A = (a_1, a_2, a_3)^{'}$ with $$a_i = \lambda_i(q_i - q_{i - 1})^{-1}, \,\, i = 1, 2, 3.$$
Easy computation shows that $A = (\frac{1}{2}, \frac{11}{6}, \frac{5}{6})^{'}$. Let $\omega = \omega_{A, Q}$ be the step-weight function associated with the two numerical vectors $A$ and $Q$. Then by Theorem \ref{66420} we have that $(S^c)_{\omega}$ defines a comonotonic approximation of $S$ such that $$E[(S^c)_{\omega}] = E[S] \,\, \text{and} \,\, Var[(S^c)_{\omega}] = Var[S].$$ In other words, $(S^c)_\omega$ defines an approximation of $S$ such that the approximation is comonotonic and moment matching up to order 2. 
\end{ex}

\section{Numerical illustrations} \label{627604}
In this section, we return to the sum of lognormal random variables as described in Section \ref{62773} and illustrate the effectiveness of Theorem \ref{66420} by using a classical numerical example. This example has been considered extensively in the literature. 
\subsection{Model description} \label{613626} 
In this subsection we describe the model and fix the notations which will be used in the following subsections to illustrate the effectiveness of Theorem \ref{66420}. We will consider the random sum $S$ which is defined by 
\begin{equation} \label{613602}
S = \sum^n_{i = 1}\alpha_ie^{Z_i}
\end{equation}
where $\alpha_1, ..., \alpha_n$ are positive real numbers and $(Z_1, ..., Z_n)^{'}$ a random vector that follows a multivariate normal distribution with 
\begin{equation*} 
\text{E}[Z_i] = -i(\mu - \frac{1}{2}\sigma^2), \,\, i = 1, 2, ..., n, 
\end{equation*} 
and 
\begin{equation*} 
\text{Cov}[Z_i, Z_j] = \min(i, j)\sigma^2, \,\, i, j = 1, 2, ..., n, \sigma >0.
\end{equation*}
This model corresponds to a stochastic return $Y_j$ in year $j$, $j = 1, 2, ..., n$, i.e., an amount of 1 at time $j - 1$ will grow to $e^{Y_j}$ at time $j$, such that $Y_1, Y_2, ..., Y_n$ are i.i.d and $N(\mu, \sigma^2)$ distributed, see \cite{Dhaene et al. (2002b)} and \cite{DVG2005} for the details. For simplicity we only examine the case where $n = 20, \mu = 0.075, \sigma = 0.15$ and $\alpha_i = 1, i = 1, 2, ..., n$. This particular choice of parameters is considered in detail in \cite{VCD2008} and \cite{VHD2005b}. Furthermore, we will fix the notations in Theorem \ref{66420} such that 
\begin{equation} \label{613654}
r = 2, m = 3 \,\, \text{and} \,\, Q = (q_1, q_2, q_3)^{'} = ({1}/{3}, {2}/{3}, 1)^{'}.
\end{equation}
One could, of course, try other choice of parameter values but the selection (\ref{613654}) of parameter values will serve our purposes well, see the following numerical illustrations. 
\subsection{Comonotonic and moment matching approximations based on $S^c$} In this subsection, the starting comonotonic approximation of the random sum $S$ in (\ref{613602}) is the comonotonic sum $S^c$ which is defined by  
\begin{equation*} 
S^c = \sum^n_{i = 1}\alpha_ie^{E[Z_i] + \sigma_{Z_i}\Phi^{-1}(U)}
\end{equation*}
where $U$ is a random variable uniformly distributed on the unit interval $(0, 1)$. 

Next we show how to obtain comonotonic and moment matching approximations of $S$ from the starting comontonic approximation $S^c$. First, we note that the moments $E[S^r]$ and $E[(S^c)^r]$ exist for any positive integer $r$ and we denote $c = (1, E[S], ..., E[S^r])^{'}$. For the model and parameters as described in Subsection \ref{613626}, one can numerically check that $\det(B) = 364.5089$ and $c = \lambda_1b_1 + \lambda_2b_2 + \lambda_3b_3$ with $$\lambda_1 = 0.1761505, \lambda_2 = 0.5758623, \lambda_3 = 0.2479872.$$ Therefore, from Corollary \ref{613595} we can conclude that the step-weighted version $(S^c)_\omega $ of $S^c$ associated with the step-weight function $\omega = \omega_{A, Q}$ with 
$$A = (0.5284514, 1.7275870, 0.7439616)^{'} \,\, \text{and} \,\, Q = ({1}/{3}, {2}/{3}, 1)^{'}$$ defines a comonotonic approximation of $S$ such that 
\begin{equation*}
E[(S^c)_\omega] = E[S] \,\, \text{and} \,\, Var[(S^c)_\omega] = Var[S]. 
\end{equation*} 
Furthermore, the quantiles and conditional tail expectations of $(S^c)_\omega$ can be calculated according to Theorem \ref{27543} and Theorem \ref{65351}, respectively. 

Table 1 and Table 2 contain the deviations of the two comonotonic approximations $S^c$ and $(S^c)_\omega$, when approximating quantiles and CTE's of the random sum $S$. They are computed according to the formulae $$\frac{\text{Q}_{p}[S_{approx}] - \text{Q}_{p}[S_{MC}]}{\text{Q}_{p}[S_{MC}]} \times 100\%$$ and $$\frac{\text{CTE}_{p}[S_{approx}] - \text{CTE}_{p}[S_{MC}]}{\text{CTE}_{p}[S_{MC}]} \times 100\%$$ respectively, where $S_{approx}$ represents one of the two approximations and $S_{MC}$ denotes the Monte Carlo simulation result which is based on generating 500,000 random paths. The estimates obtained from this time-consuming simulation will serve as benchmark. The random paths are based on antithetic variables in order to reduce the variance of the Monte-Carlo estimate. The figures displayed in bold in the tables correspond to the best approximations, which means the ones with the smallest absolute deviation, relative to the Monte-Carlo result. For comparison, the mean absolute deviation (MAD) is presented in the last column. 

From Table 1 we see that the MAD corresponding to the probability levels $p = 0.25, 0.50, 0.75, 0.90, 0.95$, which can be regarded as a global performance measure, decreases from $4.78\%$ to $3.91\%$. In addition, the MAD corresponding to the probability levels $p = 0.95, 0.975, 0.99, 0.995, 0.999$, which can be regarded as a local performance measure for high quantiles of the distribution of $S$, decreases from $12.14\%$ to $6.34\%$. Therefore, we can conclude that after being step-weighted, the approximating performance of the comonotonic sum $S^c$ has been essentially improved. Similar observations for CTE can be made from Table 2. 

\begin{figure} [!htbp] 
\centering 
\includegraphics[width = 0.75 \textwidth]{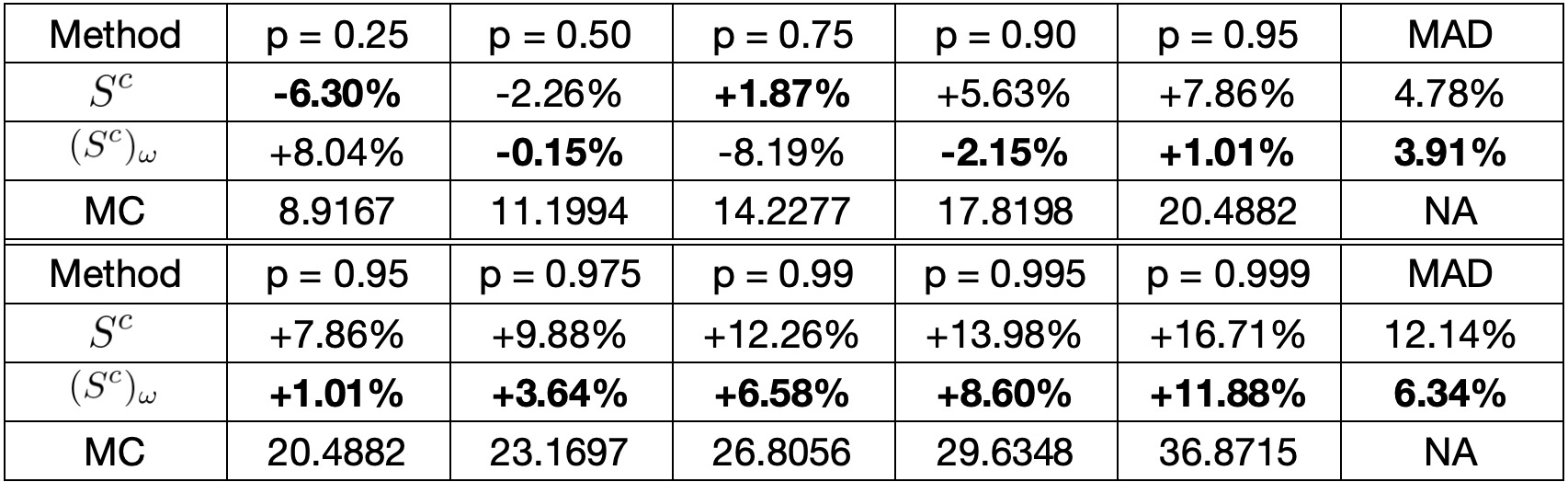} 
\caption*{Table 1: Approximations for selected quantiles of $S$ based on $S^c$ and $(S^c)_\omega$ with $\mu = 0.075, \sigma = 0.15$, $n = 20$, and yearly payments of 1.} 
\label{fig: 619672}
\end{figure}

\begin{figure} [!htbp] 
\centering 
\includegraphics[width = 0.75 \textwidth]{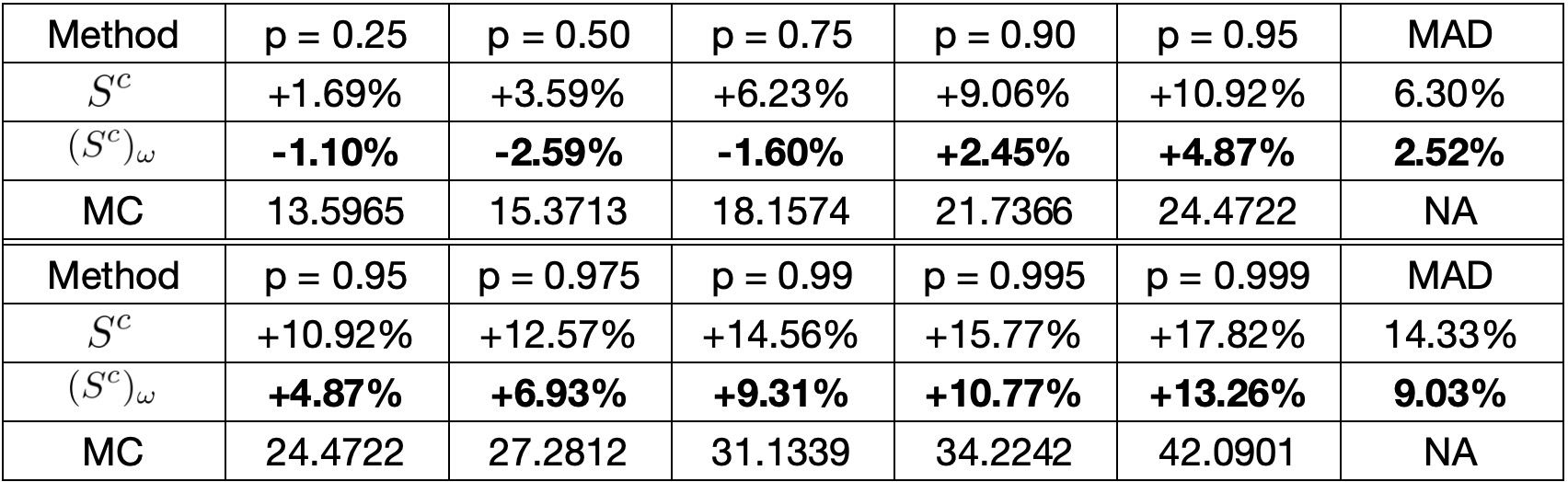} 
\caption*{Table 2: Approximations for selected CTE's of $S$ based on $S^c$ and $(S^c)_\omega$ with $\mu = 0.075, \sigma = 0.15$, $n = 20$, and yearly payments of 1.} 
\label{}
\end{figure}

\subsection{Comonotonic and moment matching approximations based on $S^l_{TB}$} In this subsection, the starting comonotonic approximation of the random sum $S$ in (\ref{613602}) is the comonotonic sum $S^l_{TB}$ which is given by  
\begin{equation*}
S^l_{TB} = \sum^n_{i = 1}\alpha_i \exp\Big(\text{E}[Z_i] + \frac{1}{2}(1 - r_i^2)\sigma_{Z_i}^2 + r_i\sigma_{Z_i}\Phi^{-1}(V)\Big)
\end{equation*}
where $V = \Phi(\frac{\Lambda - \text{E}[\Lambda]}{\sigma_{\Lambda}})$ is a random variable uniformly distributed on $(0, 1)$ and $r_i$ is the correlation coefficient between the two random variables $Z_i$ and $\Lambda = \sum^n_{j = 1}\lambda_j^{TB}Z_j$ with $\lambda_j^{TB} = \alpha_je^{\text{E}[Z_j]}, \,\, j = 1, ..., n.$ For the details see Example \ref{621193}. 

Based on the starting comonotonic approximation $S^l_{TB}$, we can obtain comonotonic and moment matching approximations of the random sum $S$. Indeed, under the assumptions in Subsection \ref{613626}, one can numerically check that $ \det(B) = 225.2262$ and $c = \lambda_1b_1 + \lambda_2b_2 + \lambda_3b_3$
with $$\lambda_1 = 0.3376989, \lambda_2 = 0.3264429, \lambda_3 = 0.3358582.$$ Therefore, from Corollary \ref{613595} we can conclude that the step-weighted version $(S^l_{TB})_\omega $ of $S^l_{TB}$ associated with the step-weight function $\omega = \omega_{A, Q}$ with 
$$A = (1.0130967, 0.9793287, 1.0075746)^{'} \,\, \text{and} \,\, Q = ({1}/{3}, {2}/{3}, 1)^{'}$$ defines a comonotonic approximation of $S$ such that 
\begin{equation*}
E[(S^l_{TB})_\omega] = E[S] \,\, \text{and} \,\, Var[(S^l_{TB})_\omega] = Var[S]. 
\end{equation*} 
The quantiles and conditional tail expectations of $(S^l_{TB})_\omega$ can be calculated according to Theorem \ref{27543} and Theorem \ref{65351}, respectively. 

Table 3 and Table 4 contain the deviations of the two comonotonic approximations $S^l_{TB}$ and $(S^l_{TB})_\omega$, together with the Monte Carlo simulation result. 
From Table 3 we can see that the MAD corresponding to the probability levels $p = 0.25, 0.50, 0.75, 0.90, 0.95$ increases from $0.06\%$ to $0.14\%$. In contrast, the MAD corresponding to the probability levels $p = 0.95, 0.975, 0.99, 0.995, 0.999$ decreases from $0.42\%$ to $0.30\%$. Therefore, we can conclude that the original comonotonic approximation $S^l_{TB}$ performs better when approximating low quantiles of $S$ while the step-weighted comonotonic approximation $(S^l_{TB})_\omega$ gives better approximations for high quantiles of $S$. Moreover, from Table 4 it can be seen that the MAD corresponding to the probability levels $p = 0.25, 0.50, 0.75, 0.90, 0.95$ decreases from $0.15\%$ to $0.09\%$ and the MAD corresponding to the probability levels $p = 0.95, 0.975, 0.99, 0.995, 0.999$ decreases from $0.89\%$ to $0.79\%$. Therefore, we can conclude that $(S^l_{TB})_\omega$ performs better when approximating CTE's of $S$. 
In other words, after being step-weighted, the CTE approximating performance of the comonotonic sum $S^l_{TB}$ has been improved. Another advantage of $(S^l_{TB})_\omega$ is that its variance is equal to the variance of the random sum $S$, facilitating interpretation when we use $(S^l_{TB})_\omega$ to approximate the unknown distribution of $S$, see Theorem \ref{66420}. 

\begin{figure} [!htbp] 
\centering 
\includegraphics[width = 0.75 \textwidth]{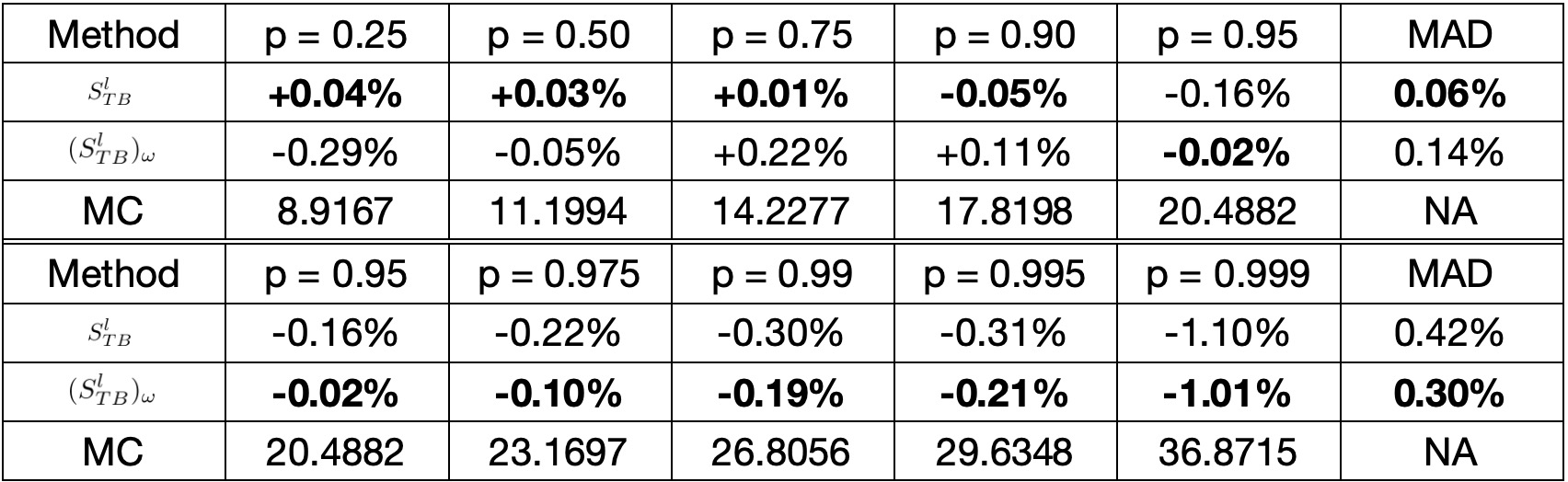} 
\caption*{Table 3: Approximations for selected quantiles of $S$ based on $S^l_{TB}$ and $(S^l_{TB})_\omega$ with $\mu = 0.075, \sigma = 0.15$, $n = 20$, and yearly payments of 1.} 
\label{fig: 619672}
\end{figure}
\begin{figure} [!htbp] 
\centering 
\includegraphics[width = 0.75 \textwidth]{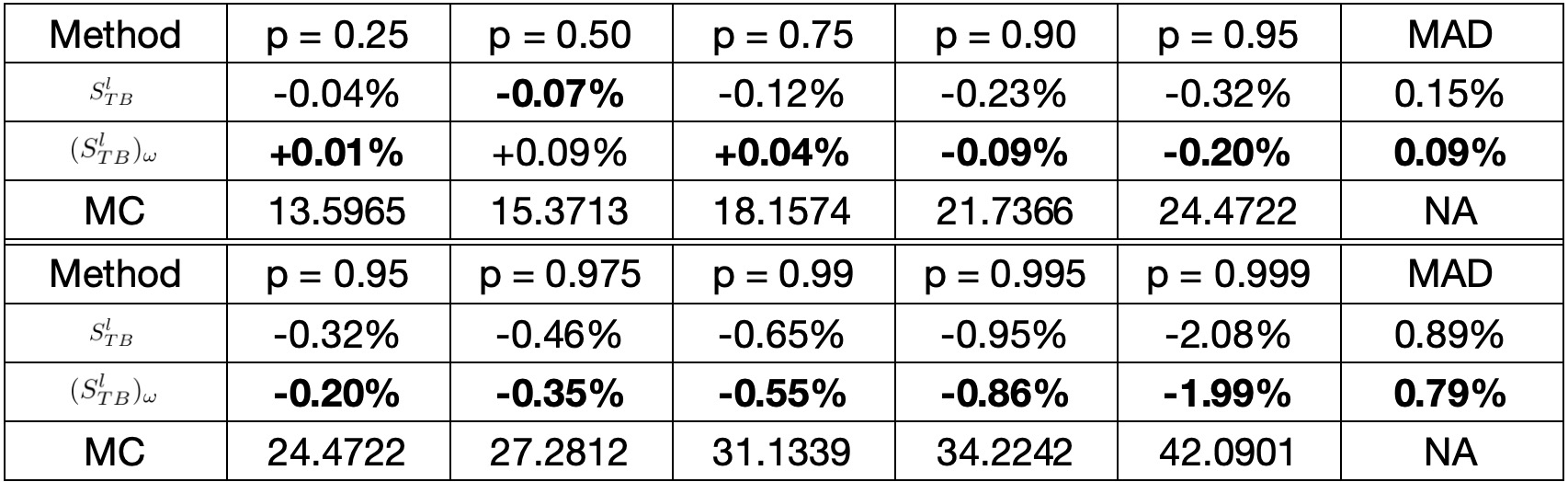} 
\caption*{Table 4: Approximations for selected CTE's of $S$ based on $S^l_{TB}$ and $(S^l_{TB})_\omega$ with $\mu = 0.075, \sigma = 0.15$, $n = 20$, and yearly payments of 1.} 
\label{}
\end{figure}

\subsection{Comonotonic and moment matching approximations based on $S^l_{MV}$} In this subsection, the starting comonotonic approximation of the random sum $S$ in (\ref{613602}) is the comonotonic sum $S^l_{MV}$ which is given by  
\begin{equation*}
S^l_{MV} = \sum^n_{i = 1}\alpha_i \exp\Big(\text{E}[Z_i] + \frac{1}{2}(1 - r_i^2)\sigma_{Z_i}^2 + r_i\sigma_{Z_i}\Phi^{-1}(V)\Big)
\end{equation*}
where $V = \Phi(\frac{\Lambda - \text{E}[\Lambda]}{\sigma_{\Lambda}})$ is a random variable uniformly distributed on $(0, 1)$ and $r_i$ is the correlation coefficient between the two random variables $Z_i$ and $\Lambda = \sum^n_{j = 1}\lambda_j^{MV}Z_j$ with $\lambda^{MV}_j = \alpha_j\text{E}[e^{Z_j}], \,\, j = 1, ..., n.$ For the details see Example \ref{621204}. 

Based on the starting comonotonic approximation $S^l_{MV}$, we can also obtain comonotonic and moment matching approximations of the random sum $S$. Indeed, under the assumptions in Subsection \ref{613626}, one can numerically check that $ \det(B) = 225.2762$ and $c = \lambda_1b_1 + \lambda_2b_2 + \lambda_3b_3$
with $$\lambda_1 = 0.3367033, \lambda_2 = 0.3280212, \lambda_3 = 0.3352755.$$ Therefore, from Corollary \ref{613595} we can conclude that the step-weighted version $(S^l_{MV})_\omega $ of $S^l_{MV}$ associated with the step-weight function $\omega = \omega_{A, Q}$ with $$A = (1.0101099, 0.9840636, 1.0058264)^{'} \,\, \text{and} \,\, Q = ({1}/{3}, {2}/{3}, 1)^{'}$$ defines a comonotonic approximation of $S$ such that 
\begin{equation*}
E[(S^l_{MV})_\omega] = E[S] \,\, \text{and} \,\, Var[(S^l_{MV})_\omega] = Var[S]. 
\end{equation*} 
The quantiles and conditional tail expectations of $(S^l_{MV})_\omega$ can be calculated according to Theorem \ref{27543} and Theorem \ref{65351}, respectively. 

Table 5 and Table 6 contain the deviations of the two comonotonic approximations $S^l_{MV}$ and $(S^l_{MV})_\omega$, 
together with the Monte Carlo simulation result. The observations made from Table 5 and Table 6 are similar to those from Table 3 and Table 4. From Table 5 we can see that the MAD corresponding to the probability levels $p = 0.25, 0.50, 0.75, 0.90, 0.95$ increases from $0.05\%$ to $0.11\%$. In contrast, the MAD corresponding to the probability levels $p = 0.95, 0.975, 0.99, 0.995, 0.999$ decreases from $0.25\%$ to $0.16\%$. Therefore, we can conclude that the original comonotonic approximation $S^l_{MV}$ performs better when approximating low quantiles of $S$ while the step-weighted comonotonic approximation $(S^l_{MV})_\omega$ gives better approximations for high quantiles of $S$. Moreover, from Table 6 it can be seen that the MAD corresponding to the probability levels $p = 0.25, 0.50, 0.75, 0.90, 0.95$ decreases from $0.11\%$ to $0.05\%$ and the MAD corresponding to the probability levels $p = 0.95, 0.975, 0.99, 0.995, 0.999$ decreases from $0.64\%$ to $0.56\%$. Therefore, we can conclude that $(S^l_{MV})_\omega$ performs better when approximating CTE's of $S$. 

\begin{figure} [!htbp] 
\centering 
\includegraphics[width = 0.75 \textwidth]{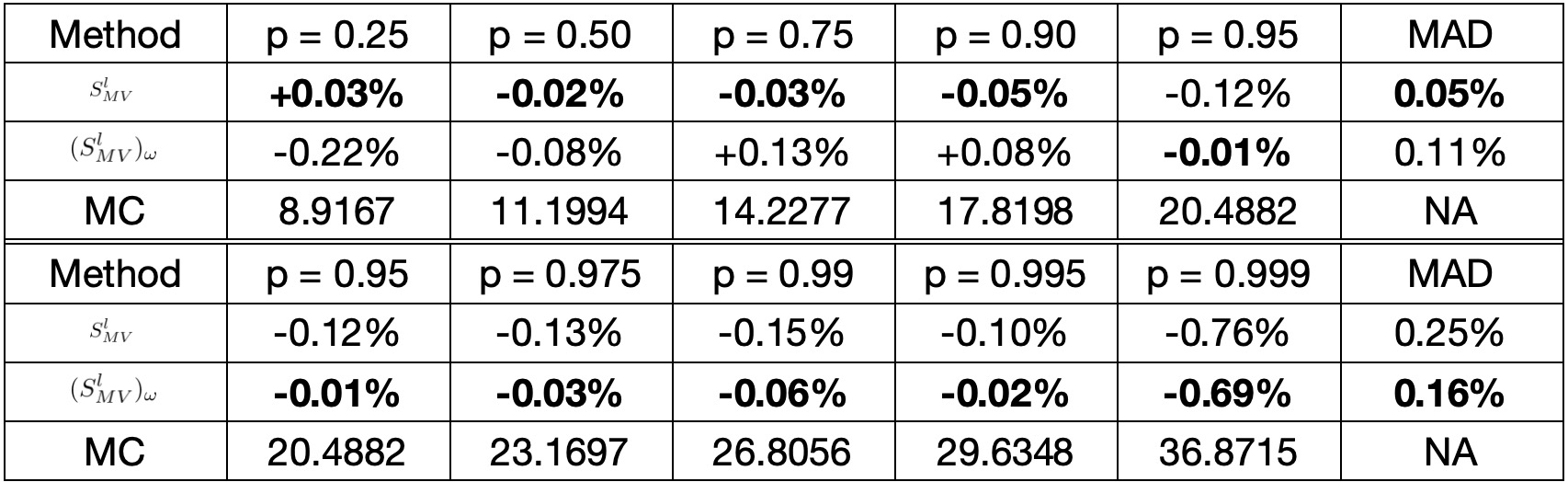} 
\caption*{Table 5: Approximations for selected quantiles of $S$ based on $S^l_{MV}$ and $(S^l_{MV})_\omega$ with $\mu = 0.075, \sigma = 0.15$, $n = 20$, and yearly payments of 1.} 
\label{fig: 619672}
\end{figure}

\begin{figure} [!htbp] 
\centering 
\includegraphics[width = 0.75 \textwidth]{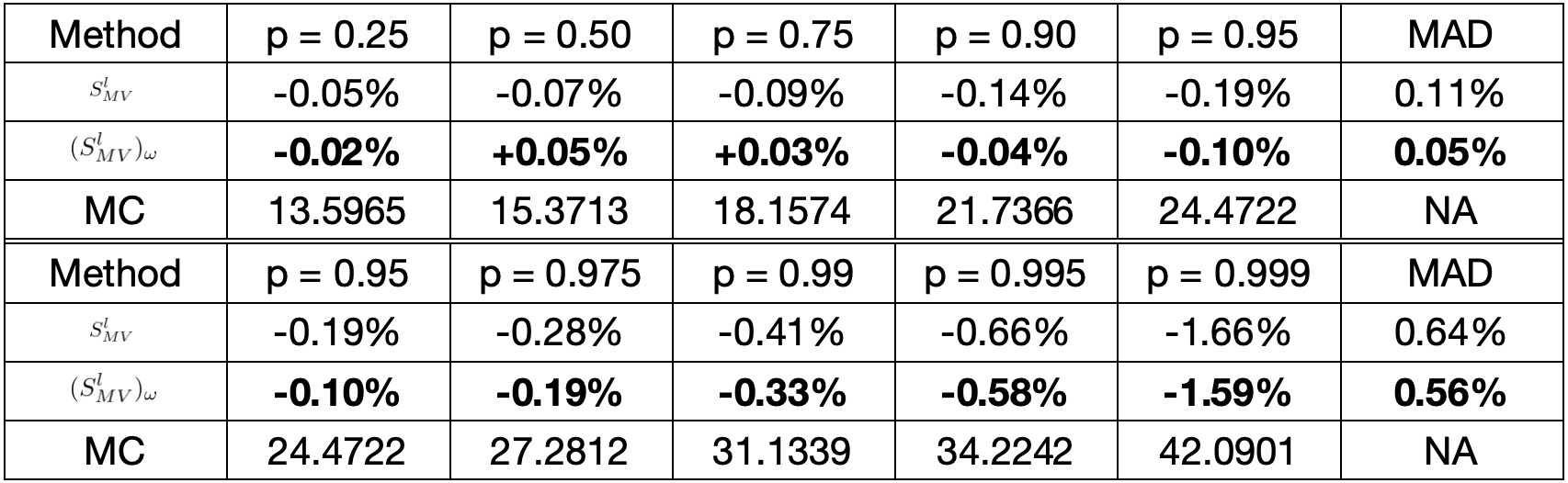} 
\caption*{Table 6: Approximations for selected CTE's of $S$ based on $S^l_{MV}$ and $(S^l_{MV})_\omega$ with $\mu = 0.075, \sigma = 0.15$, $n = 20$, and yearly payments of 1.} 
\label{}
\end{figure}

\section{Concluding remarks}
In this paper, based on the concept of weighted distribution, comonotonic and moment matching approximations for sums of lognormal random variables have been provided. The comonotonicity implies that any distortion risk measure of the newly presented approximations can easily be expressed in terms of the corresponding risk measures of the marginal terms involved while the requirement of moment matching not only facilitates interpretation but makes the new approximations close to the original sum $S$ as well. Both aspects are very important from practical point of view when considering approximations for the distribution and risk measures of the original sum $S$. Numerical results show that the approximation performance of the newly presented approximations is, overall, comparable to the classical comonotonic approximations, but in terms of the right tail of the distribution of $S$, a particularly important concept in actuarial science, our approximations perform better than the classical comonotonic ones. Another contribution of this article is the establishment of the step-weighting theory for continuous random variables. 


\begin{thebibliography}{99}
\bibitem{AGK2019} 
Arratia R, Goldstein L, Kochman F. Size bias for one and all[J]. 2019.

\bibitem{DDV2011}
Deelstra G, Dhaene J, Vanmaele M. An overview of comonotonicity and its applications in finance and insurance[J]. Advanced mathematical methods for finance, 2011: 155-179. 

\bibitem{DD2007}
Denuit M, Dhaene J. Comonotonic bounds on the survival probabilities in the Lee–Carter model for mortality projection[J]. Journal of Computational and Applied Mathematics, 2007, 203(1): 169-176.

\bibitem{DDGK2006}
Denuit M, Dhaene J, Goovaerts M, et al. Actuarial theory for dependent risks: measures, orders and models[M]. John Wiley \& Sons, 2006.

\bibitem{DDR2022}
Denuit M, Dhaene J, Robert C Y. Risk‐sharing rules and their properties, with applications to peer‐to‐peer insurance[J]. Journal of Risk and Insurance, 2022, 89(3): 615-667.

\bibitem{Denuit2019}
Denuit M. Size-biased transform and conditional mean risk sharing, with application to P2P insurance and tontines[J]. ASTIN Bulletin: The Journal of the IAA, 2019, 49(3): 591-617.

\bibitem{Denuit2020}
Denuit M. Size-biased risk measures of compound sums[J]. North American Actuarial Journal, 2020, 24(4): 512-532.

\bibitem{Dhaene et al. (2002a)}
Dhaene J, Denuit M, Goovaerts M J, et al. The concept of comonotonicity in actuarial science and finance: theory[J]. Insurance: Mathematics and Economics, 2002, 31(1): 3-33. 

\bibitem{Dhaene et al. (2002b)}
Dhaene J, Denuit M, Goovaerts M J, et al. The concept of comonotonicity in actuarial science and finance: applications[J]. Insurance: Mathematics and Economics, 2002, 31(2): 133-161. 

\bibitem{DVG2005}
Dhaene J, Vanduffel S, Goovaerts M J, et al. Comonotonic approximations for optimal portfolio selection problems[J]. Journal of Risk and Insurance, 2005, 72(2): 253-300.

\bibitem{Dhaene et al. (2006)}
Dhaene J, Vanduffel S, Goovaerts M J, et al. Risk measures and comonotonicity: a review[J]. Stochastic models, 2006, 22(4): 573-606.

\bibitem{DHL2008}
Dhaene J, Henrard L, Landsman Z, et al. Some results on the CTE-based capital allocation rule[J]. Insurance: Mathematics and Economics, 2008, 42(2): 855-863.

\bibitem{DWY2000}
Dhaene J, Wang S, Young V R, et al. Comonotonicity and maximal stop-loss premiums[J]. Bulletin of the Swiss Association of Actuaries, 2000, 2: 99-113.

\bibitem{D1990}
Dufresne D. The distribution of a perpetuity, with applications to risk theory and pension funding[J]. Scandinavian Actuarial Journal, 1990, 1990(1): 39-79.

\bibitem{FZ2008a}
Furman E, Zitikis R. Weighted premium calculation principles[J]. Insurance: Mathematics and Economics, 2008, 42(1): 459-465.

\bibitem{FZ2008b}
Furman E, Zitikis R. Weighted risk capital allocations[J]. Insurance: Mathematics and Economics, 2008, 43(2): 263-269.

\bibitem{FZ2009}
Furman E, Zitikis R. Weighted pricing functionals with applications to insurance: an overview[J]. North American Actuarial Journal, 2009, 13(4): 483-496.

\bibitem{FL2005}
Furman E, Landsman Z. Risk capital decomposition for a multivariate dependent gamma portfolio[J]. Insurance: Mathematics and Economics, 2005, 37(3): 635-649.

\bibitem{FL2008}
Furman E, Landsman Z. Economic capital allocations for non-negative portfolios of dependent risks[J]. ASTIN Bulletin: The Journal of the IAA, 2008, 38(2): 601-619.

\bibitem{KDG2000}
Kaas R, Dhaene J, Goovaerts M J. Upper and lower bounds for sums of random variables[J]. Insurance: Mathematics and Economics, 2000, 27(2): 151-168.

\bibitem{KGDD2008}
Kaas R, Goovaerts M, Dhaene J, et al. Modern actuarial risk theory: using R[M]. Berlin, Heidelberg: Springer Berlin Heidelberg, 2008.

\bibitem{KPW2012}
Klugman S A, Panjer H H, Willmot G E. Loss models: from data to decisions[M]. John Wiley \& Sons, 2012.

\bibitem{MH2011}
Mao T, Hu T. A new proof of Cheung’s characterization of comonotonicity[J]. Insurance: Mathematics and Economics, 2011, 48(2): 214-216.

\bibitem{M1997}
Milevsky M A. The present value of a stochastic perpetuity and the Gamma distribution[J]. Insurance: Mathematics and Economics, 1997, 20(3): 243-250.

\bibitem{MP1998}
Milevsky M A, Posner S E. Asian options, the sum of lognormals, and the reciprocal gamma distribution[J]. Journal of financial and quantitative analysis, 1998, 33(3): 409-422.

\bibitem{MR2000}
Milevsky M A, Robinson C. Self-annuitization and ruin in retirement[J]. North American Actuarial Journal, 2000, 4(4): 112-124.

\bibitem{PR1976}
Patil G P, Ord J K. On size-biased sampling and related form-invariant weighted distributions[J]. Sankhyā: The Indian Journal of Statistics, Series B, 1976: 48-61.

\bibitem{PR1978}
Patil G P, Rao C R. Weighted distributions and size-biased sampling with applications to wildlife populations and human families[J]. Biometrics, 1978: 179-189.

\bibitem{Qin2017}
Qin J. Biased sampling, over-identified parameter problems and beyond[M]. Singapore: Springer, 2017.

\bibitem{RS1995}
Rogers L C G, Shi Z. The value of an Asian option[J]. Journal of Applied Probability, 1995, 32(4): 1077-1088.

\bibitem{VCD2008}
Vanduffel S, Chen X, Dhaene J, et al. Optimal approximations for risk measures of sums of lognormals based on conditional expectations[J]. Journal of Computational and Applied Mathematics, 2008, 221(1): 202-218.

\bibitem{VDG2005a}
Vanduffel S, Dhaene J, Goovaerts M. On the evaluation of ‘saving-consumption’plans[J]. Journal of Pension Economics \& Finance, 2005, 4(1): 17-30. 

\bibitem{VDG2003}
Vanduffel S, Dhaene J, Goovaerts M, et al. The hurdle-race problem[J]. Insurance: Mathematics and Economics, 2003, 33(2): 405-413.

\bibitem{VHD2005b} 
Vanduffel S, Hoedemakers T, Dhaene J. Comparing approximations for risk measures of sums of nonindependent lognormal random variables[J]. North American Actuarial Journal, 2005, 9(4): 71-82. 

\bibitem{VGD2004} 
Vyncke D, Goovaerts M, Dhaene J. An accurate analytical approximation for the price of a European-style arithmetic Asian option[J]. Finance, 2004, 25(121–139): 113.

\bibitem{WD1998}
Wang S, Dhaene J. Comonotonicity, correlation order and premium principles[J]. Insurance: Mathematics and Economics, 1998, 22(3): 235-242. 

\end{thebibliography}
\end{document}